\documentclass[reqno,12pt]{amsart}%
\usepackage[body={17cm,21cm}]{geometry}

\usepackage[T1]{fontenc}
\usepackage{graphicx}
\usepackage{mathrsfs}
\usepackage{amscd,latexsym,amsthm,amsfonts,amssymb,amsmath,amsxtra}
\usepackage[colorlinks, urlcolor=blue,  citecolor=blue]{hyperref}
\usepackage{enumerate}
\usepackage[all]{xy}%
\providecommand{\U}[1]{\protect\rule{.1in}{.1in}}
\RequirePackage{amsmath}
\RequirePackage{amssymb}

\newcommand{\Z}{{\mathbb Z}}
\newcommand{\Q}{{\mathbb Q}}

\newcommand{\Br}{{\mathrm{Br}}}

\newcommand{\Cor}{{\mathrm{Cor}}}

\newcommand{\Div}{{\mathrm{Div}}}

\newcommand{\Gal}{{\mathrm{Gal}}}

\renewcommand{\Im}{{\mathrm{Im}}}

\newcommand{\Ker}{{\mathrm{Ker}}}

\newcommand{\Pic}{\mathrm{Pic}}

\renewcommand{\mod}{\ \mathrm{mod}\ }

\newcommand{\Res}{{\mathrm{Res}}}

\newcommand{\Spec}{{\mathrm{Spec}}}

\font\cyr=wncyr10
\newcommand{\Sha}{\hbox{\cyr X}}

\newcommand{\et}{{\operatorname{\acute{e}t}}}

\numberwithin{equation}{section}

\theoremstyle{remark}

\newtheorem{defi}{\rm{\textbf{Definition}}}[section]

\newtheorem{exam}[defi]{\rm{\textbf{Example}}}

\newtheorem{rem}[defi]{\rm{\textbf{Remark}}}

\theoremstyle{plain}

\newtheorem{thm}[defi]{\rm{\textbf{Theorem}}}

\newtheorem{cor}[defi]{\rm{\textbf{\textbf{Corollary}}}}

\newtheorem{lem}[defi]{\rm{\textbf{Lemma}}}

\newtheorem{prop}[defi]{\rm{\textbf{\textbf{Proposition}}}}

\begin{document}

\title{The unramified Brauer groups of normic bundles }

\author{Dasheng Wei}

\address{Hua Loo-Keng Key Laboratory of Mathematics, Academy of Mathematics and System Science, CAS, Beijing 100190, P.\ R.\ China \emph{and} School of mathematical Sciences, University of  CAS, Beijing 100049, P.R.China}

\email{dshwei@amss.ac.cn}

\date{September 22, 2022}

\begin{abstract}
We produce a partial compactification of the variety given by $P(t)=N_{K/k}(\mathbf z)$
whose Brauer group coincides with the unramified Brauer group, where $K$ is an \'etale $k$-algebra and $P(t)\in k[t]$ is a nonconstant polynomial. Then we obtain a systematic  method to compute the unramified Brauer group
for all such varieties.
 \end{abstract}


\maketitle

\section{Introduction} \label{sec.notation}

The present article focuses on the variety defined over the ground field
$k$ by an equation
\begin{equation}\label{eq:norm}
	P(t) = N_{K/k}({\bf z}),
\end{equation}
where $K$ is an \'etale $k$-algebra, $N_{K/k}$ denotes the norm
map, $\mathbf z$ is  a ``variable'' in $K$ and $P(t)\in k[t]$ is a
nonconstant polynomial.

Suppose $K/k$ is a finite extension of number fields,  for a smooth and proper model of (\ref{eq:norm}), Colliot-Th\'el\`ene conjectured that its rational points are dense in its Brauer-Manin set (see \cite{CT03}).
Colliot-Th\'el\`ene's conjecture has been extensively studied. It is known in the case where $P(t)$ is constant \cite{San81}; if additionally $K/k$ is cyclic or of prime degree, the Hasse principle and weak approximation hold. Depending on Colliot-Th{\'e}l{\`e}ne and Sansuc's descent theory \cite{CTS87},   other known
cases of Colliot-Th\'el\`ene's conjecture, in some cases leading to a
proof of the Hasse principle and weak approximation, include the class
of Ch\^atelet surfaces ($[K:k]=2$ and $\deg(P(t)) \le 4$) \cite{CTSSD87a,CTSSD87b}, a class of singular cubic hypersurfaces ($[K:k]=3$ and $\deg(P(t)) \le 3$) \cite{CTS89}, and the case where $K/k$ is arbitrary and $P(t)$ is split over $k$ with at
most two distinct roots \cite{CHS,HBS02,SJ11}, and
the case where $[K:k]=4$ and $\deg(P(t))=2$ with $P(t)$ irreducible
over $k$ and split over $K$ \cite{BHB11,DSW14}, and the case where $k=\Bbb Q$, $K/\Q$ is arbitrary and $\deg(P(t))=2$ \cite{DSW14}  and the case where $k=\Bbb Q$, $K/\Q$ is arbitrary and $P(t)$ is split over $\Q$ \cite{BM,BMS14,HSW}. Finally, under Schinzel's hypothesis, this conjecture is known for $P(t)$ arbitrary and either $K/k$ cyclic \cite{CTSSD98} or of prime degree and almost abelian \cite{HSW,Wei14}. The integral points of the equation (\ref{eq:norm})  have also been studied in \cite{CWX18,DW17}. See \cite[Example 3.12 and \S 3.3.4]{Wit18} for a more detailed discussion of these results.

\medskip

To compute the Brauer-Manin set of (\ref{eq:norm}), one must compute its unramified Brauer
group. If $[K:k]=2$, a smooth and proper model of (\ref{eq:norm}) can been easily constructed, so we can compute its unramified Brauer group. Using a similar idea in the case $[K:k]=2$, V\'arilly-Alvarado and Viray \cite{VV12,VV15} constructed a smooth and proper model of (\ref{eq:norm}) for the special case when $P(t)$ is a separable polynomial of degree $dn$ or $dn-1$ and $K/k$ is a cyclic field extension of degree $n$. However, in generally, it is hard to construct a smooth and proper model of (\ref{eq:norm}). In 2003, Colliot-Th\'el\`ene, Harari and Skorobogatov \cite{CHS}
produced a partial compactification of (\ref{eq:norm}). For this partial compactification, they gave a formula for its vertical Brauer group and the quotient of its Brauer group
by the vertical Brauer group. Using this formula,  they computed  the
unramified Brauer group for some special cases (e.g., $P(t)$ is split over $k$ with two distinct roots whose multiplicities are relatively prime). However, the Brauer group of the partial compactification is generally larger than the unramified Brauer group (e.g., see \cite[Theorem 3.6 and Proposition 3.7]{Wei14}). Based on Colliot-Th\'el\`ene, Harari and
Skorobogatov's work, the author \cite{Wei14} studied  the case that $P(t)$ is irreducible and $K/k$ is abelian;  Derenthal, Smeets and the author \cite{DSW14} studied the case that $\deg(P(t))=2$ and $[K:k]=4$. 
The problem of producing partial
compactifications of (\ref{eq:norm}) whose Brauer group coincides with the unramified
Brauer group remained open, the goal of this note is to solve this problem
in full generality.
%In this note, we produce a partial compactification of (\ref{eq:norm}) whose Brauer group coincides with  the unramified Brauer group.
%As an application, we compute the unramified Brauer group for the case $P(t)$ irreducible and $K/k$ Galois.
% However, it is still open how to determine the unramified Brauer group for
%more general cases.

The paper is organized as follows. In Section 1, by blowing-up along some closed subsets of the singular locus, we construct a partial compactification of (\ref{eq:norm}) and prove that its Brauer group is equal to the unramified Brauer group. In Section 2, applying this partial compactification, we give an explicit formula of the unramified Brauer group of the case $P(t)$ irreducible and $K/k$ Galois, furthermore, if either $K$ is the composite of two linearly disjoint fields or the Galois group $\Gal(K/k)$ is abelian or dihedral, we compute the unramified Brauer group.

\bigskip
{\bf Terminology.} Let $k$ be a field of characteristic zero and $\Gamma_k$ its absolute Galois group. Throughout this text, intersections of fields and composites of fields are taken inside the given algebraic closure $\bar k$.
Let~$Z$ be a variety over $k$.
The Brauer group $\Br(Z)=H^2_{\et}(Z,\Bbb G_m)$
contains the \emph{algebraic Brauer group} $\Br_1(Z)=\Ker\big(\Br(Z) \to \Br(Z \otimes_k \bar k)\big)$,
and the subgroup $\Br_0(Z) = \mathrm{Im}\big(\Br(k)\to\Br(Z)\big)$ of constant classes.
When~$Z$ is smooth, we denote $\Br_{\mathrm {un}}(Z) \subset \Br(Z)$ to be
the \emph{unramified Brauer group}, which coincides with the Brauer group of any smooth and proper
model of $Z$.

\bigskip
{\bf Acknowledgements. } I am really grateful to Jean-Louis Colliot-Th\'el\`ene for introducing this problem to me and many valuable suggestions on an earlier version of this paper, to Olivier Wittenberg for useful comments. This work is supported by National Key R$\&$D Program of China.

\section{The construction of partial compactification} \label{sec2}

Let $k$ be a field of characteristic zero. Let $U_1$  be the affine variety  in $ \Bbb A^{n+1}_k$ defined by (\ref{eq:norm}), where $P(t) \in k[t]$ is a polynomial of degree $m$, $K$ is an \'etale $k$-algebra of degree $n$ and $N_{K/k}$ is a norm form for the extension $K/k$.

 We may assume $P(0)\neq 0$, otherwise we  may replace $t$ by $t+a$, here $a\in k$ such that $P(a)\neq 0$. Let $0 \leqslant s < n, s\equiv -m \mod n$. Let $U_2$ be the affine variety defined by the equation
 \begin{equation}\label{eq:2} {t'}^s \tilde P(t')=N_{K/k}({\bf z'}),
 \end{equation}
where $\tilde P(t')={t'}^mP(1/t')$. 
There is a birational map $U_2 \dashrightarrow U_1, (t', {\bf z})\mapsto (\frac{1}{t'}, \frac{{\bf z'}}{t'^j})$ with $j=(m+s)/n$. Let  $U_0$  be the open subvariety of $U_2$ defined by $t'\neq 0$. This birational map yields an open immersion $U_0\longrightarrow U_1$.
We glue $U_1$ and $U_2$ along the open subset $U_0$, then we obtain the variety $U$ with a surjective projection $\pi: U \longrightarrow \mathbb{P}^1$ by $(t,\mathbf z) \mapsto t$. Let $\infty$ be the projection of the point $(0, \mathbf z) \in U_1$, the fiber $\pi^{-1}(\infty)$ is the smooth variety defined by $0\neq P(0)= N_{K/k}(\mathbf z).$

In this whole section, we will always assume $n\mid \text{deg}(P(t))$, otherwise we may replace $P(t)$ by $t^s\tilde P(t), s\equiv -m \mod n, 0\leqslant s < n$, $i.e.$, we replace (\ref{eq:norm}) by (\ref{eq:2}).  The variety $U$ has a surjective projection $\pi : U\to  \mathbb P^1$ which is smooth at the point $t=\infty$.

We write $P(t)=c p_1(t)^{e_1}\cdots p_s(t)^{e_s}$, where $c\neq 0$  and $p_1(t),\cdots, p_s(t)$ are distinct irreducible polynomials. We may assume each $e_i<n$, otherwise we may replace $e_i$ by a positive integer $a$ with $a\equiv e_i \mod n$. Over $\bar k $, the affine variety $\overline U_1$ can be described as an equation of the form
 \begin{equation}\label{eq:kbar}
 	 P(t)=z_1\cdots z_n.
 \end{equation}
The singular locus of $\overline U_1$ is the closed subset $$\bigcup_{e_i>1}\bigcup_{j<\mu}\{(t,z_1,\cdots z_n): p_i(t)=z_j=z_\mu=0\}.$$
%
%A subset of $ \{z_1,\cdots,z_n\}$ is called a $d$-cyclic orbit if its cardinality is $d$ and it is an orbit of some $g\in \Gal(\bar k/k)$.

For every exponent $e_i$ in $P(t)$, define $e'_i=\text{gcd}(e_i,n)$.  For each $e'_i>1$, we define the set of integers
$$\{d\in \mathbb Z: d>1 \text{ and } d\mid e_i'\}.$$
We list these integers in this set by $e_i'=d_{i,1}>d_{i,2}>\cdots> d_{i,t_i}>1$, where $t_i$ is the numbers of factors of $e_i'$ which are larger than $1$ and we denote $d_{i,t_{i+1}}=1$.

We will construct the new variety by blowing-up along some ideal sheaves. By reordering the factors $p_i(t)^{e_i}$ of $P(t)$, we may assume  $$e_i'>1 \text{ if } 1 \leqslant i \leqslant r; e_i'=1 \text{ if } r < i \leqslant s.$$
For $1 \leqslant i \leqslant r$, we denote the closed subset of $X$ (defined over $k$) by
$$W_{i,j} =\bigcup_{1\leq i_1< i_2<\cdots< i_j\leq n}\{p_i(t)=z_{i_1}=z_{i_2}=\cdots=z_{i_j}=0\}.$$
For $i=1$, $d_{1,1}=e_1'$, we blow $V^{(0)}:=U\setminus W_{1,e_1'+1}$ along the ideal sheaf $I_1$ defined by the ideal $$\bigcap_{1\leq i_1< i_2<\cdots< i_{e_1'}\leq n}(p_1(t)^{e_1/e_1'}, z_{i_1},z_{i_2},\cdots,z_{i_{e_1'}}),$$
then we get a variety $\widetilde{V}^{(0)}$.

Over $\bar k$, the ideal sheaf $I_1$  on $V^{(0)}$ is the sum of the ideal sheaves of $$(p_1(t)^{e_1/e'_1},z_{i_1},z_{i_2},\cdots,z_{i_{e_1'}} ),$$ whose support is the disjoint union of $\{t=\eta_j,z_{i_1}=z_{i_2}=\cdots=z_{i_{e_1'}}=0 \}$ with $p_1(\eta_j)=0$.
%and $\eta$ runs through all roots of $p_1(t)=0$.
Since $V^{(0)}$ is geometrically integral, the variety  $\widetilde{V}^{(0)}$ is also geometrically integral by \cite[Chapter 2,Propsition 7.16]{Ha}.
The blowing-up  $\widetilde{V}^{(0)}$ of $V^{(0)}$ along $(p_1(t)^{e_1/e'_1},z_{i_1},z_{i_2},\cdots,z_{i_{e_1'}} )$  is the subvariety of $\Bbb A^{n+1}\times \Bbb P^{e'_1}$ defined by equations (\ref{eq:kbar}) and
$$
	\frac{p_1(t)^{e_1/e'_1}}{t_0}=\frac{z_{i_1}}{t_1}=\cdots= \frac{z_{i_{e_1'}}}{t_{e_1'}},
$$
where $(t,z_{i_1}\cdots,z_{i_{e_1'}})\in \Bbb A^{n+1}$ and $(t_0:t_1:\cdots:t_{e_1'})\in \Bbb P^{e'_1}$. Locally, we may set $t_0=1$, then $z_{i_j}=p_1(t)^{e_1/e'_1} t_j$ for $1\leq j\leq e'_1 $. Then the open subvariety $\{ t_0\neq 0 \}$ of $\widetilde{V}^{(0)}$  is defined by
\begin{equation}\label{eq:t0}
	cp_2(t)^{e_2}\cdots p_s(t)^{e_s}=t_1\cdots t_{e'_1}\prod_{j\not \in\{i_1,i_2,\cdots,i_{e_1'}\}}z_j.
\end{equation}
%which is geometrically integral. Similarly, for $t_i\neq 0$ and $i\geq 1$, the open subvarieties of the blowing-up are also integral.
Let $t=\eta$ be a root of $p_1(t)=0$. Set $t=\eta$ in (\ref{eq:t0}) then we get an exceptional divisor of open subvariety $\{ t_0\neq 0 \}$ of $\widetilde{V}^{(0)}$.
It is clear that the intersection of the exceptional divisors in $\widetilde{V}^{(0)}$ with $\{ t_0=0 \}$ has dimension $n-2$ (codimension 2 in $\widetilde{V}^{(0)}$), hence exceptional divisors in $\widetilde{V}^{(0)}$ are isomorphic to $$\mathbb Z[L_1/k] \otimes \{M: M\subset Z_K,|M|=e_1'\}$$ as $\Gamma_k$-modules, here we denote $L_i=k[t]/(p_i(t))$, $\Z[L_i/k]=\Z[\Gamma_k/\Gamma_{L_i}]$ for any $i$ and $Z_K$ is the disjoint union of $\Gamma_k/\Gamma_{K_i}$ if we write $K=K_1\times K_2\times \cdots\times K_u$ with each $K_i$ is a field.
%Therefore, the intersection of the exceptional divisors in $\widetilde{V}^{(0)}$ with $\{ t_0\neq 0 \}$ are  are isomorphic to %$$\mathbb Z[L_1/k] \otimes \{T: T\subset \Gamma_k/\Gamma_K,|T|=e_1'\}$$ as $\Gamma_k$-modules, here $L_1=k[t]/(p_1(t))$.

 Let $V^{(1)}$ be the subvariety of $\widetilde{V}^{(0)}$ obtained by removing  the intersection of the singular locus of $\widetilde{V}^{(0)}$ with its exceptional divisors.  Obviously, $U\setminus W_{1,e_1'}$ can be viewed as an open subvariety of $V^{(1)}$. The intersection of the singular locus of $V^{(1)}$ with $\{p_1(t)=0\}$ is just $ W_{1,2}\setminus W_{1,e'_1}$. Therefore $ W_{1,2}\setminus W_{1,e'_1}$ is a closed subset of $V^{(1)}$. Then $W_{1,d_{1,2}+1}\setminus W_{1,e_1'}$ is also a closed subset of   $V^{(1)}$.  Removing $W_{1,d_{1,2}+1}\setminus W_{1,e_1'}$ in $V^{(1)}$, we get an open subvariety of $V^{(1)}$, which is  denoted by $V_{1,1}$.

%We remove the intersection of the exceptional divisors and the singular locus of $V_{1,1}$ in $V_{1,1}$, we also denote it by $V_{1,1}$. The singular locus of $V_{1,1}$ is contained in $U_{1,2+1}\setminus U_{1,e_1'} \subset V_{1,1}$.   We remove the intersection of singular locus and $U_{i,d_{1,2}+1}$

Let $I_2$ be the ideal sheaf of $V_{1,1}$ which is uniquely determined by the $\Gamma_k$-invariant ideal $$\bigcap_{i_1< i_2<\cdots< i_{d_{1,2}}}(p_1(t)^{e_1/d_{1,2}}, z_{i_1},z_{i_2},\cdots,z_{i_{d_{1,2}}}),$$ whose support is $W_{1,d_{1,2}}\setminus W_{1,d_{1,2}+1}$.
We blow $V_{1,1}$ along the ideal sheaf $I_2$ and obtain a variety. Similarly, removing  the intersection of the singular locus of this variety with  exceptional divisors, then we get the variety denoting by $V^{(2)}$.

Over $\bar k$, the ideal sheaf $I_2$  on $V^{(1)}$ is also the sum of the ideal sheaves of $$((t-\eta)^{e_1/d_{1,2}},z_{i_1},z_{i_2},\cdots,z_{i_{d_{1,2}}} ),$$ whose support is a disjoint union. By a similar argument as above, the divisor groups generated by exceptional divisors  are isomorphic to $\mathbb Z[L_1/k] \otimes \{M: M\subset Z_K,|M|=d_{1,2}\}$ as $\Gamma_k$-modules.

%Similarly, in $V^{(2)}$, we remove the intersection of the singular locus of $V^{(2)}$ with the exceptional divisors from $I_2$.
%Then the singular locus of $V^{(2)}$ is contained in $U\setminus W_{1,d_{1,2}}$, which is just $W_{1,2}\setminus W_{1,d_{1,2}}$.
 Similarly, removing $W_{1,d_{1,3}+1}\setminus W_{1,d_{1,2}}$ in $V^{(2)}$, we get an open subvariety of $V^{(2)}$  denoted by $V_{1,2}$. For $d_{1,3},\cdots, d_{1,t_1}$, we continue the process as above. Finally, we get a variety $V_{1,t_1}$.

%replace $V^{(2)}$ by $V^{(2)}$ removing the intersection of the singular locus of $V^{(2)}$ with the exceptional divisors from $I_2$;

Note that $U\setminus \{p_1(t)=0\}$ can be viewed as an open subvariety of $V_{1,t_1}$, therefore, for $i=2,\cdots, r$, $i.e.$, $p_2(t)^{e_2},\cdots, p_r(t)^{e_r}$, we continue the above process of blow-ups similar as $p_1(t)^{e_1}$, finally we get a variety $V_{r,t_r}$. Note that $V_{r,t_r}$ is a smooth variety over $k$ since we remove $W_{i,t_{i+1}+1}=W_{i,2}$ for each $i$.

Denote $U'=U\setminus \{P(t)=0\}$, any nonempty fiber of $\pi: U'\to \mathbb P^1$ is a principal homogeneous space of the torus $T$, here $T$  is defined by the equation $ N_{K/k}({\bf z})=1.$ Let $T^c$ be a fixed $T$-invariant smooth compactification of $T$ (see \cite{CHS1}). Let $U'\times^{T} T^c$ be the contracted product and it has an open subvariety $U'\times^{T} T\cong U'$. By our construction, $V_{r,t_r}$ is a system of blow-ups of $U$ and $U'$ can be viewed as an open subvariety of $V_{r,t_r}$. 
We glue $V_{r,t_r}$ and $U'\times^{T} T^c$ along the isomorphism $U'\times^{T} T\cong U'$, we obtain a new smooth scheme over $k$ which is separated over $k$.
%We obtain 
%\begin{equation}
%	\text{a variety over } k \text{ by gluing } V_{r,t_r} \text{ and } U'\times^{T} T^c \text{ along the isomorphism } U'\times^{T} T\cong U'
%\end{equation}
% which is separated over $k$.
\emph{In the sequel, }
	\begin{equation}\label{def:X'} 
	\text{we shall denote } X' \text{ the variety thus constructed.}
	\end{equation}

 We can also construct a simpler variety than $X'$. In the above blow-up for each $p_i(t)^{e_i}, 1\leq i\leq r$, we only do the first step (and removing its singular locus), i.e., we do not blow-up for the smaller divisors of $e'_i$ than itself. Then we glue the variety similar as $V_{r,t_r}$ and $U'\times^{T} T^c$ along the isomorphism $U'\times^{T} T\cong U'$, we obtain a new smooth scheme over $k$.
\emph{In the sequel},	
 \begin{equation} \label{def:X}
 	\text{we shall denote } X \text{ the variety thus constructed.}
 \end{equation}  It is clear that $X$ is an open subset of $X'$. Such $X$ will be useful when $K/k$ is a Galois field extension. 

Recall that $U_1\subset U$ is the affine variety defined by (\ref{eq:norm}). Let $V$ be the maximal smooth open subset of $U_1$. We can compare $X'$ (or $X$) with the partial compactification $V^{\mathrm{CTHS}}$ in \cite{CHS}.  Let $U_1'=U_1\setminus \{P(t)=0\} $. Colliot-Th\'el\`ene, Harari and
Skorobogatov \cite{CHS}  constructed the partial compactification of (\ref{eq:norm}) by gluing $U_1'\times^{T} T^c$ and $V$ along $U_1'\times^{T} T\cong U_1'$, where $T^c$ is the fixed $T$-invariant smooth compactification of $T$.  They proved that $\bar k[V^{\mathrm{CTHS}}]^\times=\bar k^\times$, $\Pic(\overline {V}^{\mathrm{CTHS}})$ is torsion free and $\Br(\overline {V}^{\mathrm{CTHS}})=0$, see \cite[Proposition 2.3]{CHS}. Obviously $V^{\mathrm{CTHS}}$ is a Zariski open subset of $X$ (or $X'$), hence we have the following lemma.
% (In the next section, we will always assume $K/k$ is a Galois field extension.)

\begin{lem}\label{lem:constant}
	The groups  $\Pic(\overline X)$ and $\Pic(\overline X')$  are torsion free,  $\bar k[X]^\times=\bar k[X']^\times=\bar k^\times$ and $\Br(\overline X)=\Br(\overline X')=0$.
\end{lem}

We write $K=K_1\times \cdots\times K_u$ where each $K_i$ is a field, recall that $Z_K$ is the disjoint union of $\Gamma_k/\Gamma_{K_i}$.  Now we consider the divisors in $\overline{X'}\setminus \overline {U_1'}$ and $\overline{X}\setminus \overline {U_1'}$:

i) horizontal divisors  which are  divisors of $\Div(\overline {U'}\times^{\overline T} \overline T^c)$ with support out of $\overline {U'}$, we denote it by $\text{Div}_h$. By \cite[Lemma 2.1]{CHS}, $\text{Div}_h$ is isomorphic to $\Div_{\overline T^c\setminus \overline T}( \overline T^c)$ as $\Gamma_k$-modules. 
%in $X_\eta$, $\eta$ is the generic point of $\mathbb P^1$,

ii) vertical divisors in $ \overline{U}\setminus \overline{U'_1}$, which is isomorphic to $\mathbb Z[K/k]\otimes \mathbb Z_P\oplus \mathbb Z$, where $\mathbb Z$ comes from the divisor at $\infty$, $\mathbb Z[K/k]:=\mathbb Z[Z_K]$ and $\mathbb Z_P=\oplus_i \Bbb Z[L_i/k]$ where $L_i=k[t]/(p_i(t))$.

iii) exceptional divisors, which is isomorphic to
\begin{equation}\label{S}
S:=\bigoplus_{i=1}^r  \mathbb Z[L_i/k]\otimes\mathbb Z[\Omega_i],
\end{equation}
   where \begin{equation}\label{omega_1}\Omega_i=\bigcup_{d>1,d\mid e_i'}\{M: M\subset Z_K, \# M=d\}
\end{equation} for $\overline{X'}$, and
\begin{equation}\label{omega_2}\Omega_i= \{M: M\subset Z_K, \# M=e'_i\}
\end{equation}
for $\overline{X}$.

Denote
\begin{align}\label{def:D }
&\mathbb D=\mathbb Z_P\otimes \mathbb Z[K/k]\oplus \mathbb Z \oplus S,\\
\notag
&\text{Div}_{\overline{X'}\setminus\overline{U'_1}} (\overline{X'})=\text{Div}_h\oplus \mathbb D.
\end{align}

 Let $N_K=\sum_{\sigma\in Z_K}\sigma$ and $\Upsilon=\sum_{i=1}^{r}\sum_{M\in \Omega_i} M\in S$. We define a morphism $f: \mathbb Z_P \rightarrow \mathbb D $ by sending any $\tau\in \Gamma_k/\Gamma_{L_i}$ to $(\tau\otimes N_K,-1,\tau\otimes\Upsilon)$. Let $\widehat{T'}$ be the quotient of $f$ which is torsion-free.
Note that $\Pic(\overline{U'_1})=0$, we have the following commutative diagram with exact sequences (a similar commutative diagram has already appeared in \cite[Proposition 2.2]{CHS})
 \begin{equation} \label{pic}
\xymatrix{
	 & 0 \ar[d] & 0 \ar[d] & 0 \ar[d] \\
	0\ar[r] & \mathbb Z_P \ar[r]^{f}\ar[d] & \mathbb D \ar[r] \ar[d] & \widehat{T'}   \ar[r]\ar[d] & 0 \\
	0\ar[r] & \bar k[U'_1]^\times/\bar k^\times  \ar[r]^{\text{div}} \ar[d] & \text{Div}_{\overline{X'}\setminus\overline{U'_1}} (\overline{X'}) \ar[r] \ar[d] & \Pic(\overline {X'}) \ar[r]\ar[d]  & 0\\
	0\ar[r] & \widehat{T}  \ar[r]\ar[d] &\text{Div}_h \ar[r] \ar[d] & \Pic(\overline{T^c})   \ar[r]\ar[d]  & 0 \\
	 & 0    & 0  & 0.
}
\end{equation}
In the diagram (\ref{pic}), we may replace $X'$ by $X$.
%Hence there is the exact sequence
%
%$$0\rightarrow \widehat {T'} \rightarrow \Pic (\overline{X'}) \rightarrow \Pic (\overline{T^c})\rightarrow 0.$$

By the diagram (\ref{pic}), we may construct a morphism $\widehat T\rightarrow \widehat T'$, such that  the following diagram  of exact sequences is commutative
\begin{equation}\label{rrr}
\xymatrix {
	0\ar[r] &\widehat {T} \ar[r]\ar[d]^{j}  &\text{Div}_h \ar[r] \ar[d]   &\Pic (\overline{T^c})   \ar[r]\ar@{=}[d] & 0 \\	
	0\ar[r]  &\widehat {T'}  \ar[r] & \Pic (\overline{X'}) \ar[r]  &\Pic (\overline{T^c})  \ar[r]  & 0,
}
\end{equation}
where $j: \widehat {T}\rightarrow \widehat {T'} $ is induced by the map $$\mathbb Z[K/k]\rightarrow \mathbb D, \sigma \longmapsto (-\sum_{i=1}^s e_i N_i \otimes \sigma, m/n,-S_\sigma)$$
where $\sigma\in \Gamma_k/\Gamma_K$, $N_i=\sum_{\tau\in \Gamma_k/\Gamma_{L_i}}\tau $ and $$
S_\sigma =\sum_{i=1}^r N_i\otimes \sum\limits_{d>1,d\mid e_i'}\frac{e_i}{d}\sum\limits_{\tiny\begin{array}{c}
\sigma \in M\subset Z_K
\\ \# M =d
\end{array} } M.$$
If we replace $X'$ by $X$, the diagram (\ref{rrr}) also holds by modifying the definition of $S_\sigma$ by 
$$\sum\limits_{i=1}^r N_i\otimes \frac {e_i}{e_i'}\sum\limits_{\tiny\begin{array}{c}
		\sigma \in M\subset \Gamma_k/\Gamma_K
		\\ \#M =e'_i
\end{array} } M.$$

\begin{lem} \label{lem:seq}
We have exact sequences
\begin{align}\label{seq:main}
	&0\rightarrow   H^1(k,\widehat{T'})/j^*(H^1(k,\widehat{T})) \rightarrow H^1(k,\Pic(\overline{X'}))  \rightarrow \Ker[\Sha^2_\omega(\widehat{T})\rightarrow \Sha^2_\omega(\widehat{T'})]\rightarrow 0;\\	
	\label{seq:main'}
	&0\rightarrow   H^1(k,\widehat{T'})/j^*(H^1(k,\widehat{T})) \rightarrow H^1(k,\Pic(\overline{X}))  \rightarrow \Ker[\Sha^2_\omega(\widehat{T})\rightarrow \Sha^2_\omega(\widehat{T'})]\rightarrow 0.
\end{align}	
\end{lem}
\begin{proof}By (\ref{rrr}),
we have the commutative diagram
 \begin{equation*} %\label{pic'}
\xymatrix{
	H^0(k,\Pic(\overline{T^c})) \ar@{=}[d] \ar[r] &  H^1(k,\widehat{T})\ar[d]^{j^*} \ar[r] &  H^1(k,\text{Div}_h)\ar[r] \ar[d] & H^1(k,\Pic(\overline{T^c})) \ar[r] \ar@{=}[d] & H^2(k,\widehat{T})   \ar[d] \\
	H^0(k,\Pic(\overline{T^c}))  \ar[r] &  H^1(k,\widehat{T'})\ar[r] &  H^1(k,\Pic(\overline{X'})\ar[r]  & H^1(k,\Pic(\overline{T^c})) \ar[r] & H^2(k,\widehat{T'})
}
\end{equation*}
Since $\text{Div}_h$ is a permutation module, one has $H^1(k,\text{Div}_h)=0$. Since $H^1(k,\Pic(\overline{T^c}))\cong \Sha^2_\omega(\widehat{T})$, one obtains the exact sequence (\ref{seq:main}) (similarly for (\ref{seq:main'})).
\end{proof}

\begin{rem} For any field $k\subset F\subset \bar k$, we also have exact sequences
\begin{align}\label{seq:main-ass}
	&0\rightarrow   H^1(F,\widehat{T'})/j^*(H^1(F,\widehat{T})) \rightarrow H^1(F,\Pic(\overline{X'}))  \rightarrow \Ker[\Sha^2_\omega(\widehat{T_F})\rightarrow \Sha^2_\omega(\widehat{T'_F})]\rightarrow 0;\\
	\label{seq:main-ass'}
		&0\rightarrow   H^1(F,\widehat{T'})/j^*(H^1(F,\widehat{T})) \rightarrow H^1(F,\Pic(\overline{X}))  \rightarrow \Ker[\Sha^2_\omega(\widehat{T_F})\rightarrow \Sha^2_\omega(\widehat{T'_F})]\rightarrow 0.
\end{align}
\end{rem}

%We may choose $g\in \Gal(\bar k(D)/k(D))$, let $L'=\overline{k}^g$. Over $L'$, then we have $\Sha^2_\omega(L',\widehat{T})=0 $. Then
%$$H^1(L',\Pic(\overline{X})\cong H^1(L',\widehat{T'})/j^*(H^1(L',\widehat{T})). $$
%
%
%Then $$K\otimes_k L'\cong K_1\oplus\cdots \oplus K_m.$$

Recall that $U_1$ is the affine variety defined by (\ref{eq:norm}) and that $U_1'\subset U_1$ is defined by $P(t)\neq 0$.
We write $P(t)=cp_1(t)^{e_1}\cdots p_{s}(t)^{e_s}$ is a product of irreducible polynomials over $k$.
Let $L_i=k[t]/(P(t))$, $\mathbb Z_P=\oplus_i \mathbb Z[L_i/k]$ and $\eta_i$ the class of $t$ in $L_i$. We define
\begin{equation}\label{def:alpha}
\alpha: \mathbb Z_P \rightarrow \bar k[U_1']^\times,\Gamma_{L_i} \in \mathbb Z[L_i/k] \mapsto t-\eta_i.
\end{equation}
Since $H^2(k,\mathbb Z_P) \cong \oplus_i H^1(L_i,\Q/\Z)$ and $\Br_1(U_1')=H^2(k,\bar k[U_1']^*)$, $\alpha$ induces a morphism $$\alpha':  \oplus_i H^1(L_i,\Q/\Z) \rightarrow \Br_1(U_1').$$  In fact, we do not need to assume that $n\mid \deg(P(t))$ in the above argument and the following lemma.

\begin{lem}
	\label{lem:alpha}
	The map $\alpha'$ sends $(\chi_i)_i \in \oplus_{i=1}^s H^1(L_i,\Bbb Q/\mathbb Z)$
	to $$\sum_{i}\Cor_{L_i/k}(t-\eta_i,\chi_i) \in \Br_1(U_1').$$
\end{lem}

\begin{proof} We only need to show $\alpha'(\chi_i)=\Cor_{L_i/k}(t-\eta_i,\chi_i)$ for each $i$.
	Let $f:U_1'\to \Spec(k)$, $f':U_1' \otimes_k L_i \to \Spec(L_i)$ and $\nu:U_1' \otimes_k L_i \to U_1'$ be natural morphisms.
	Let $\gamma:\Bbb Z \to \Bbb G_m$ be the morphism of \'etale sheaves on $U_1' \otimes_k L_i$
	by sending 1 to $t-\eta_i \in L_i[U_1']^\times$.
	Applying the functors $H^2(U_1,-)$ and $H^2(k,f_*-)$ to the morphisms of \'etale sheaves on $U_1'$
	\begin{align*}
		\xymatrix{
			\nu_*\Bbb Z \ar[r]^{\nu_*\gamma} & \nu_*\Bbb G_m \ar[r]^{N_{L/k}} & \Bbb G_m\rlap{,}
		}
	\end{align*}
	we obtain a commutative diagram
	\begin{align*}
		%\label{diag:alpha'}
		\begin{aligned}
			\xymatrix@C=2.4em@R=3ex{
				\ar@{=}[d] H^2(k,f_*\nu_*\Bbb Z) \ar[r] & H^2(k,f_*\nu_*\Bbb G_m) \ar[r] & H^2(k,f_*\Bbb G_m) \ar@{=}[d] \\
				*!<3.35em,0ex>\entrybox{H^1(L_i,\Bbb Q/\Bbb Z)=H^2(L_i,\Bbb Z)}
				\ar[d]_{f'^*} \ar[rr]^{\Res(\alpha')} &&
				H^2(k,\bar k[U_1']^*)
				\ar[d] \\
				H^2(U_1' \otimes_k L_i,\Bbb Z) \ar[r]^(.52){\chi \mkern3mu\mapsto\mkern3mu (t-\eta_i) \cup \chi} &
				\Br(U_1' \otimes_k L_i) \ar[r]^(.56){\Cor_{L_i/k}} & \Br(U_1')\rlap.
			}
		\end{aligned}
	\end{align*}
	The lemma then follows from the commutativity of this diagram.
\end{proof}

Note that $\Br_1(X')/\Br_0(X')\cong H^1(k,\Pic(X'))$, hence (\ref{seq:main}) gives the morphism $$\beta:H^1(k,\widehat T') \rightarrow \Br_1(X')/\Br_0(X').$$
We denote $$\Psi_i:=\Ker[H^1(L_i,\Q/\mathbb Z)\rightarrow  H^1(L_i\otimes K,\Q/\mathbb Z)\oplus H^1(F,\mathbb Z)\oplus H^2(F,S\otimes \Q/\Z)].$$
By the first row of the diagram (\ref{pic}), one obtains
\begin{equation}\label{key}	
	H^1(k,\widehat T') \cong \Ker[H^2(k,\mathbb Z_P)\rightarrow H^2(k,\mathbb Z_P\otimes\mathbb Z[K/k]\oplus\mathbb Z\oplus S)] \cong\oplus_{i=1}^s \Psi_i.
\end{equation}
Therefore, $\beta$ induces the map $\beta':\oplus_{i=1}^s \Psi_i \rightarrow \Br_1(X')/\Br_0(X').$ In fact, the remark for  \cite[Proposition 2.5]{CHS}  gives the following explicit description of $\beta'$ and here we provide a short proof.

\begin{lem} \label{lem:rem}
	$\beta'$ is defined by send any $(\chi_i)_i\in \oplus_{i=1}^s \Psi_i$ to $$\sum_{i}\Cor_{L_i/k}(t-\eta_i,\chi_i) \in \Br_1(X').$$
\end{lem}
\begin{proof}
For any field $k$ with $char(k)=0$ and any smooth variety $Z$,  the classical exact sequence $$0\rightarrow \bar k(Z)^*/\bar k \rightarrow \Div(\overline Z) \rightarrow \Pic(\overline Z)\rightarrow 0$$
gives an exact sequence \begin{equation}\label{seq:pic}
0\rightarrow H^1(k, \Pic(\overline Z))\rightarrow H^2(k,\bar k(Z)^*/\bar k) \rightarrow H^2(k,\Div(\overline Z)).
\end{equation}
Using Grothendieck's purity theorem (cf.  \cite[Theorem 1.3.2]{CTSD94}), (\ref{seq:pic}) induces the natural map $ H^1(k,\Pic(\overline Z)) \rightarrow\Br_1(Z)/\Br_0(Z)$ which can be viewed as the inverse map in Hochschild-Serre's spectral sequence by the observation 3 of \cite[Section 2, p. 122]{Li69}.

The top two rows of the diagram (\ref{pic}) give  the commutative diagram
\begin{equation} \label{diag: h1pic}
\xymatrix{
	&H^1(k,\widehat{T'}) \ar[d] \ar[r] &  H^2(k,\Z_P) \ar[d] \\
	 &H^1(k,\Pic(\overline{X'}))  \ar[r] & H^2(k,\bar k[U_1']^\times/\bar k),
}
\end{equation}
where the right-side vertical map is induced by $\alpha$ in (\ref{def:alpha}), the lemma then follows from the commutativity of the diagram (\ref{diag: h1pic}) and Lemma \ref{lem:alpha}.	
\end{proof}
\begin{rem}
If we replace $X'$ by $X$, by the same  arguments as above, Lemma \ref{lem:rem} also holds  for $X$.
\end{rem}

Now we can prove the main theorem of this article.
\begin{thm} \label{thm:main}
	The Brauer group
	$$\Br_1(X')=\Br(X')=\Br_\mathrm{un}(X').$$
	If $K/k$ is a Galois field extension, $$\Br_1(X)=\Br(X)=\Br_\mathrm{un}(X).$$
\end{thm}
\begin{proof}
	Since $\Br(\overline{X})=\Br(\overline{X'})=0$ by Lemma \ref{lem:constant}, we have $\Br_1(X)=\Br(X)$ and $\Br_1(X')=\Br(X')$.
% For any $g\in \Gal(\bar k(D)/k(D))$, let $L'=\overline{k}^g$.
	
	 Let $A$ be a discrete valuation ring containing $k$, and with the fraction field $k(X')$ and the residue field $\kappa_A$. Let  $\partial_A: \Br(k(X'))\rightarrow
	H^1(\kappa_A,\mathbb Q/\mathbb Z)$ be the residue map. By Grothendieck's purity theorem (cf. \cite[Theorem 1.3.2]{CTSD94}), we have
	$$\Br_\mathrm{un}(X')=\bigcap_{A}\Ker(\partial_A)\subset \Br(k(X')),$$
	where $A$ runs through all discrete valuation rings as above. Therefore,	for any $\mathcal B \in \Br_1(X')$, we want to show  $\mathcal B \in \Br_\mathrm{un}(X')$, it suffices to prove the triviality of
	$\partial_A(B)$ for any such discrete valuation ring $A$, $i.e.$, $\partial_A(B)(g)=0$ for any $g \in
	\Gal(\bar \kappa_A/\kappa_A)$.
	
	We extend the embedding $k \subset
	\kappa_A$ to an embedding $\bar k \subset \bar \kappa_A$, then $g$
	acts on $\bar k$. Let $F=\bar k^g$, then $\Gal(\bar k/F)$ is pro-cyclic.
	
	Since $X'$ is geometrically integral and $char(k)=0$, the completion of
	$k(X')$ for the given valuation is isomorphic to $\kappa_A((\pi))$,
	where $\pi$ is a uniformizer.  Considering the valuation given by $\pi$
	on $(F. \kappa_A)((\pi))$ and using $F(X') = k(X) \otimes_k
	F$, this defines a discrete valuation on $F(X')$ with valuation ring $A_{F}$ whose residue field
	$\kappa_{A_{F}}=\kappa_A . F$, hence $g\in \Gal(\bar \kappa_A/\kappa_{A_{F}})$, in order to show $\partial_A(B)(g)=0$, we only need to show $\partial_{A_{F}}(B)(g)=0$.
	
	Over $F$, we have $\Sha^2_\omega(\widehat T_{F})=0$ since $\bar k/F$ is pro-cyclic by the definition of $F$. By (\ref{seq:main-ass}), we have
	\begin{equation}\label{eq:Br}
		\Br_1(X'_{F})/\Br_0(X'_{F})\cong H^1(F,\Pic(\overline{X'}))\cong H^1(F,\widehat{T'})/j^*(H^1(F,\widehat{T})). 
	\end{equation}
	
	Let $P(t)=cq_1(t)^{f_1}\cdots q_{s'}(t)^{f_{s'}}$ be a product of irreducible polynomials over $F$ and let $F_i=F[t]/(q_i(t))$, then
	\begin{align}
		& H^1(F,\widehat{T'})\cong \Ker[H^2(F,\mathbb Z_P)\rightarrow H^2(F,\mathbb Z_P\otimes\mathbb Z[K/k]\oplus \mathbb Z\oplus S))] \notag \\
\label{ker}
&\cong\Ker[\oplus_{i=1}^{s'}H^1(F_i,\Q/\mathbb Z)\rightarrow  H^1(F\otimes K,\Q/\mathbb Z)\oplus H^1(F,\Q/\mathbb Z)\oplus H^1(F,S\otimes \Q/\Z)].
		\end{align}

	Let $(\chi_i)_i \in \oplus_i H^1(F_i,\Q/\mathbb Z)$ be contained in the kernel (\ref{ker}), then we have $\sum_i\text{Cor}_{F_i/F}(\chi_i)=0$. By Lemma \ref{lem:rem} (replacing $k$ by $F$), such $(\chi_i)_i$ gives an element $\mathcal B=\sum_i \text{Cor}_{F_i/F}(t-\eta_i,\chi_i)\in \Br(X')$, we only need to show the triviality of the residue of $\mathcal B$ by (\ref{eq:Br}).

	Let $A'$ be a discrete valuation of $F(X')$. Assume the valuation $v_{A'}(t)=l<0$, then $v_{A'}(q_i(t))=l \deg(q_i(t))$. Therefore, the residue $$\partial_{A'}(\mathcal B)=\sum_i \partial_{A'}(\text{Cor}_{F_i/F}(\frac{t-\eta_i}{t},\chi_i))+\sum_i\partial_{A'}(\text{Cor}_{F_i/F}(t,\chi_i))=0+l \sum_i \text{Cor}_{F_i/F}(\chi_i)=0.$$
	
	Assume $v_{A'}(t)\geq 0$. If all $v_{A'}(q_i(t))=0$, obviously the residue $\partial_{A'}(\mathcal B)=0.$  So we only need to consider the case  $v_{A'}(q_i(t))>0$ for some $i=i_0$. Since all $q_i(t)$ are relative prime, one obtains
	$v_{A'}(q_j(t))=0$ for any $j\neq i_0$. Therefore $$\partial_{A'}(\mathcal B)=\partial_{A'}(\text{Cor}_{F_{i_0}/F}(t-\eta_{i_0},\chi_{i_0})).$$
 Let $A'_{i_0}$ be a discrete valuation of $F_{i_0}(X')$ such that $\text{ord}_{A'_{i_0}}(t-\eta_{i_0})\neq 0$ and which restricted to $F (X')$ is just $A'$.
%	Then by the relatively prime of $t-\eta$ and $q_{i_0}(t)(t-\eta)^{-1}$, we have $\text{ord}_{A_{i_0}}(t-\eta)=v_A(q_{i_0}(t))$.
%	 Let $\chi_{i_0}$. be the restriction of the character $\chi'_{i_0}$ of $\Gamma_{F_i}$.
	 Hence $$\partial_{A'}(\mathcal B)=v_{A'_{i_0}}(t-\eta_{i_0})\chi_{i_0}.$$
Let  the exponent of $q_{i_0}(t)$ in $P(t)$ is $e_i$, $i.e.$, $q_{i_0}(t)\mid p_i(t)$.
	
	We write \begin{equation}\label{field} K\otimes_k F_{i_0}\cong K'_1\oplus\cdots \oplus K'_{m'},
	\end{equation}
	 where all $K'_j/F_{i_0}$ are cyclic field extensions which are contained in a cyclic extension. By the equation $P(t)=N_{K'_1/F_{i_0}}(\mathbf z_1)\cdots N_{K'_{m'}/F_{i_0}}(\mathbf z_{m'})$, one obtains $\frac{\gcd(n'_1,\cdots, n'_{m'})}{\gcd(e_i, n'_1,\cdots, n'_{m'})}|v(t-\eta_{i_0})$ since  $v_{A'_{i_0}}(N_{K'_i/F_{i_0}}(\mathbf z_{m'}))$ is divided by $n'_i$,  where $n'_i=[K'_i:F_{i_0}]$ . In the following, we only need to show that the order of $\chi_{i_0}$ divides $\frac{\gcd(n'_1,\cdots, n'_{m'})}{\gcd(e_i, n'_1,\cdots, n'_{m'})}$, which implies $\partial_A(\mathcal B)=0$. If we replace $X'$ by $X$, the above argument also holds.
	
	(1) The case $X'$:
	
	Let $d=(e_i, n'_1,\cdots, n'_{m'})$ be a factor of $e'_i=(e_i,n)$. For any $1\leq j\leq m'$, let $H_j$ be the subgroup of order $d$ of the cyclic group $\Gal(K'_j/F_{i_0})$. Then $H_j$ gives an element $\gamma_j$ of $\Omega_i$ (see (\ref{omega_1}) for definition), the $\gamma_j$ generates a permutation $\Gal(K_j/F_{i_0})$-module which is isomorphic to $\mathbb Z[G_j/H_j]$ and is a direct summand of $\mathbb Z[\Omega_i]$, where $G_j=\Gal(K'_j/F_{i_0})$. By  (\ref{ker}), the order of $\chi_{i_0}$ is a factor of $n'_j/d $. When $j$ runs through $1,2, \cdots, m' $, the order of $\chi_{i_0}$ is a factor of $\gcd(n'_1,\cdots,n'_{m'})/d $. Then we complete the proof for $X'$.
	
	(2) The case $X$ (here $K/k$ is a Galois field extension):
	
	Since $K/k$ is  Galois, all $K'_i$ in (\ref{field}) are isomorphic, hence $n':=n'_1=\cdots=n'_{m'}$. Let $d=:\gcd(e_i, n'_1,\cdots, n'_{m'})=\gcd(e_i,n')=\gcd(e'_i,n')$.  Let $H_i$ be the unique subgroup of order $d$ in $\Gal(K'_i/F_{i_0})$ for $1\leq i \leq r$. Let $r'_i:=e'_i/d=e'_i/\gcd(e'_i,n') $, $r'_i$ is a factor of $n/n'$, hence $r'_i \leq m'$. Then the disjoint union $\cup_{j=1}^{r'_i} H_i\subset \Gamma_k/\Gamma_K$ gives an element $\gamma$ of $\Omega_i$ (see (\ref{omega_2})), the $\gamma$ generates a permutation $G'$-module which is isomorphic to $\mathbb Z[G'/H']$  and it is a direct summand of $\mathbb Z[\Omega_i]$, where $G'\cong\Gal(K'_1/F_{i_0})\cong \cdots \cong \Gal(K'_{m'}/F_{i_0})$ and $H'\cong H_1\cong \cdots \cong H_m$. By the definition of  (\ref{ker}), the order of $\chi_{i_0}$ is a factor of $n'/d $. Then we complete the proof for $X$. \qedhere

\end{proof}

\begin{rem}
	In fact, for any field $L$ over $k$, by similar arguments as above we can show $$\Br_1(X'_L)=\Br(X'_L)=\Br_\mathrm{un}(X'_L);$$	
	if $K/k$ is a Galois field extension, $$\Br_1(X_L)=\Br(X_L)=\Br_\mathrm{un}(X_L).$$
\end{rem}

Combining Lemma \ref{lem:seq} with Theorem \ref{thm:main}, we have the following corollary.
\begin{cor}	\label{cor:seq}
We have the exact sequence
	\begin{align}\label{seq:main2}
		0\rightarrow   H^1(k,\widehat{T'})/j^*(H^1(k,\widehat{T})) \rightarrow \Br_\mathrm{un}(X')/\Br_0(X')  \rightarrow \Ker[\Sha^2_\omega(\widehat{T})\rightarrow \Sha^2_\omega(\widehat{T'})]\rightarrow 0.
	\end{align}	
			If $K/k$ is a Galois field extension, we have the exact sequence
	\begin{align}\label{seq:Br-Galois}
		0\rightarrow   H^1(k,\widehat{T'})/j^*(H^1(k,\widehat{T})) \rightarrow \Br_\mathrm{un}(X)/\Br_0(X)  \rightarrow \Ker[\Sha^2_\omega(\widehat{T})\rightarrow \Sha^2_\omega(\widehat{T'})]\rightarrow 0.
	\end{align}
\end{cor}

\begin{exam}\label{lem:3.1} Let $X$ be the variety (\ref{def:X}) defined by the equation $P_1(t)P_2(t)=N_{K/k}(\mathbf{z})$, where $P_1(t)$ and $P_2(t)$ are irreducible and of degree 2. Let $L_i=k[t]/(P_i(t))$. Suppose that $L_1\neq L_2$ and $K=L_1.L_2$. Then $$\Br_\mathrm{un}(X)=\Br_0(X),\text{ but } \Br(V^\mathrm{CTHS})/\Br_0(V^\mathrm{CTHS})=\Z/2.$$
\end{exam}

\begin{proof} It is clear that  $\Br(V^\mathrm{CTHS})/\Br_0(V^\mathrm{CTHS})=\Z/2$ by \cite[Propsition 2.5]{CHS}.

	Since $P_1(t)P_2(t)$ is separable and of degree $[K:k]=4$, we have $S=0$ in (\ref{S}).
	Let $H_i =\Gal(L_i/k), i=1,2$ and $G=\Gal(K/k)=H_1\times H_2$. By Corollary \ref{cor:seq}, we only need to show that this map $\Sha^2_\omega(\widehat{T})\rightarrow \Sha^2_\omega(\widehat{T'})$ is injective. Since $\widehat{T}$ and $\widehat{T'}$ is split by $K$, it suffices to show that  $H^2(G,\widehat{T}) \to H^2(G,\widehat{T'})$ is injective. By the Hochschild-Serre spectral sequence $H^i(H_1,H^j(H_2, -))\Rightarrow H^{i+j}(G,-)$, since $E_2^{2,0}(\widehat{T})=E_2^{0,2}(\widehat{T})=0$, we have the commutative diagram
	\begin{equation*}
		\xymatrix{
			&H^2(G,\widehat{T})\ \ \  \ar[r]  \ar[d]  &\ \   H^1(H_1,H^1(H_2,\widehat{T}))\ar[d]^g \\ &H^2(G,\widehat{T'})\ \ \ \ar[r] &   H^1(H_1,H^1(H_2,\widehat{T'}))
		}
	\end{equation*}
	and the first horizontal map is injective. It is easy to show that $H^1(H_2,\widehat{T})\cong H^1(H_2,\widehat{T'})\cong \Bbb Z/2$, which implies that $g$ is injective. Therefore $H^2(G,\widehat{T}) \to H^2(G,\widehat T')$ is injective.
\end{proof}

If $k= \Q$, $K/\Q$ is an arbitrary extension of number fields and $P(t)$ is split over $\Q$, then the Brauer-Manin obstruction controls the rational points of  (\ref{eq:norm}) entirely (see \cite{BM,BMS14,HSW}). The following result gives a simple description of the Brauer group in this direction.

\begin{exam}\label{exam2} Let $X$ be the variety (\ref{def:X}) defined by the equation (\ref{eq:norm}), where $K/k$ is a finite Galois field extension of degree $n$ and $P(t)=c(t-a_1)^{e_1}\cdots (t-a_s)^{e_s}$ with all $e_i\in k$ are distinct. Let $e_i'=\mathrm{gcd}(e_i,n)$. Suppose $n\mid \deg(P(t)) \text{ and } \text{gcd}(e_1,\cdots,e_s,n)=1$, then $\Br_\mathrm{un}(X)$ is generated by $$\rho+ \sum_{i=1}^s (t-a_i,\chi_i),$$ where $\rho\in \Br_0(X)$, $\chi_i\in H^1(K/k,\Q/\Z)$ satisfies $\sum_{i=1}^s \chi_i=0$ and $\chi_i(g)=0$ for any $g\in \Gal(K/k)$ with $g^{e_i'}=1_G$.
\end{exam}
\begin{proof} The quotient of $\widehat{T'}$ by $\Z\oplus S$ (in (\ref{def:D })) coincides with $\widehat{T}\otimes \Z_P\cong  \widehat{T}^s$, then the map $f: H^2(K/k,\widehat{T})\rightarrow H^2(K/k,\widehat{T'})$ induces the map $f': H^2(K/k,\widehat{T})\rightarrow H^2(K/k,\widehat{T})^s$ which sends $\alpha\in H^2(K/k,\widehat{T})$ to $(-e_1\alpha, \cdots,-e_s\alpha)$. Since $\text{gcd}(e_1,\cdots,e_s,n)=1$, one obtains $f'$ is injective, hence $f$ is also injective. Since $\widehat{T}$ and $\widehat T'$ is split by $K$, the injectivity of $f$ implies $\Ker[\Sha^2_\omega(\widehat{T})\rightarrow \Sha^2_\omega(\widehat{T'})]=0$.  The proof then follows from Corollary \ref{cor:seq} and Lemma \ref{lem:rem}.
%	
%There is a natural quotient map $\widehat{T'} \rightarrow \widehat{T}\otimes \Z_P$, which implies $$\Ker[\Sha^2_\omega(\widehat{T})\rightarrow \Sha^2_\omega(\widehat{T'})]\subset \Ker[\Sha^2_\omega(\widehat{T})\rightarrow \Sha^2_\omega(\widehat{T}\otimes \Z_P)].$$
%since $\text{gcd}(e_1,\cdots,e_s,n)=1$, one obtains the map $\Sha^2_\omega(\widehat{T})\rightarrow \Sha^2_\omega(\widehat{T}\otimes \Z_P)$ is injective. Hence . 	
\end{proof}

%\begin{lem}
%	Assume that $L_i\subset K$ and $K/k$ is Galois, then $\Br(X)/\Br_0(X)\cong H^1(k,\Pic(\overline {\widetilde U}))$.	
%\end{lem}
%\begin{proof}
%	
%	
%	
%	
%	We have the commutative diagram
%	$$0\rightarrow \widehat{T} \rightarrow\widehat{T'}\rightarrow  \Pic(\widetilde U_K)\rightarrow 0,$$
%	which gives the long exact sequence
%	\begin{equation}\label{seq:main2}
%		0\rightarrow   H^1(K/k,\widehat{T'})/j^*(H^1(K/k,\widehat{T})) \rightarrow H^1(K/k,\Pic(\widetilde U_K)  \rightarrow \Ker[H^2(K/k,\widehat{T})\rightarrow H^2(K/k,\widehat{T'})]\rightarrow 0.
%	\end{equation}
%	
%	Since $\Sha^2_\omega(\widehat{T})\cong H^2(K/k,\widehat{T})$, we have
%	\begin{equation}\label{seq:main3}
%		0\rightarrow   H^1(K/k,\widehat{T'})/j^*(H^1(K/k,\widehat{T})) \rightarrow H^1(K/k,\Pic(X_K)  \rightarrow \Ker[H^2(K/k,\widehat{T})\rightarrow H^2(K/k,\widehat{T'})]\rightarrow 0.
%	\end{equation}	
%	Since $\Pic(\overline X)$ and $\Pic(\overline{\widetilde U})$ is torsion-free, then $H^1(K/k,\Pic(\widetilde U_K))\cong H^1(k,\Pic(\overline {\widetilde U}))$ and  $H^1(k,\Pic(\overline X))\cong H^1(K/k,\Pic(X_K))$. Comparing (\ref{seq:main2}) and (\ref{seq:main3}), then  $\Br(X)/\Br_0(X)\cong H^1(k,\Pic(\overline X))\cong  H^1(k,\Pic(\overline {\widetilde U}))$.
%\end{proof}

\section{The case $P(t)$ irreducible and $K/k$ Galois}

In the section,  we always assume that $K/k$ is a Galois field extension of degree $n$ and that $P(t)$ is a separable polynomial of degree $m$. Let $X$ be the variety (\ref{def:X}) defined by the equation (\ref{eq:norm}). We write $P(t)=p_1(t)\cdots p_s(t)$ with all $p_i(t)$ are irreducible. Let $L_i=k[t]/(p_i(t))$ and $e'=\gcd(n,m)$. We do not assume~$n\mid m$ in this section.

 Suppose $n\nmid m$.  let $U_1'$ be the  open affine subvariety of $X$ defined by (\ref{eq:norm}) and $P(t)\neq 0$. We may assume $P(0)\neq 0$. Let $0 \leqslant \delta < n, \delta\equiv -m \mod n$ and $\delta'=(m+\delta)/n$. Then $X$ contains an affine open subvariety $U_2$ defined by the equation
\begin{equation*} {t'}^\delta \tilde P(t')=N_{K/k}({\bf z'}),
\end{equation*}
where $\tilde P(t')={t'}^mP(1/t')$. Let $U_2'\subset U_2$ be defined by $\tilde P(t')\neq 0$. 

In \S \ref{sec2}, in fact we  obtain the diagram (\ref{pic}) by using the inclusion $U_2'\subset X$ when $n\nmid m$; hence $\mathbb D$ in (\ref{def:D }) contains the infinite divisor which is not very convenient when we compute  cohomology groups of $\widehat T'$. So, to obtain  a new "simple" diagram similar with (\ref{pic}) we will use the inclusion $U_1'\subset X$ (no matter whether $n\mid m$ or not) and  redefine \begin{equation}\label{D}
	\mathbb D=\mathbb Z_P\otimes \mathbb Z[K/k] \oplus S',
\end{equation}
where $S'=\Z[K/k]\oplus \mathbb Z[\Omega]$ with $\Omega=\{M: M\subset \Gamma_k/\Gamma_K, |M|=e'\}$ if $n \nmid m$; $S'=\mathbb Z$ if  $n \mid m$. If $n\nmid m$, such $\mathbb D$ is simpler than (\ref{def:D }). 
Similarly, the new diagram similar with (\ref{pic})  is also commutative, its rows and columns are exact.
The $\Gamma_k$-module $\widehat{T'}$ is the Galois lattice defined by the following exact sequence \begin{equation}\label{seq:new T'}
0 \rightarrow \Z_P\xrightarrow{f}  \mathbb D \to \widehat{T'} \rightarrow 0
\end{equation}
where $f$ sends any $\tau_i \in \Gamma_k/\Gamma_{L_i}$ to
$$f(\tau_i)=\begin{cases}(\tau_i\otimes N_K,-N_K, -S_K) &\text{ if } n\nmid m,\\
	(\tau_i\otimes N_K,-1) &\text{ if	} n \mid m,
\end{cases} $$ with
 $$S_K=\sum\limits_{\tiny\begin{array}{c}  M\subset \Gamma_k/\Gamma_K,
		\\ \# M =e'\end{array}} M \text{ and }  N_K=\sum_{\sigma\in \Gamma_k/\Gamma_{K}}\sigma.$$
In the diagram (\ref{rrr}), we redefine the morphism 
 $j: \widehat {T}\rightarrow \widehat {T'} $ induced by the morphism $\mathbb Z[K/k]\rightarrow \mathbb D$ by sending $\sigma\in \Gamma_k/\Gamma_{K}$  to $(-\sum_{i=1}^s N_i \otimes \sigma,j_\sigma),$
where \begin{align}\label{def:j-sigma}
	&j_\sigma=\begin{cases} (\delta'N_K-\delta\sigma,\delta' S_K-\frac {\delta}{e'}S_\sigma) &\text{ if  } n\nmid m,\\	
	 1,&\text{ if  } n\mid m;
\end{cases}\\
\notag
   &N_i=\sum_{\tau\in \Gamma_k/\Gamma_{L_i}}\tau \text{ and } S_\sigma =  \sum_ { 
	\tiny\begin{array}{c}\sigma \in M\subset \Gamma_k/\Gamma_K,
		\\ \# M =e'\end{array}
} M .
\end{align} 
Then we obtain a commutative diagram which is analogous to (\ref{rrr}).

Combining with Theorem \ref{thm:main}, by similar arguments as for (\ref{seq:main'}) in the proof of Lemma (\ref{lem:seq}),  we obtain the exact sequence
\begin{equation}\label{seq:main-1}
0\rightarrow   H^1(k,\widehat{T'})/j^*(H^1(k,\widehat{T})) \rightarrow \Br_\mathrm{un}(X)/\Br_0(X)  \rightarrow \Ker[\Sha^2_\omega(\widehat{T})\rightarrow \Sha^2_\omega(\widehat{T'})]\rightarrow 0.
\end{equation}

	%$N_{L/L'}=\sum_{\tau\in \Gamma_{L'}/\Gamma_L} \tau$

Let $L'_i=L_i\cap K$,  $\mathbb Z_{P'}=\oplus_i \Bbb Z[L'_i/k]$ and \begin{equation*}\label{D'}
	\mathbb D'=\mathbb Z_{P'}\otimes \mathbb Z[K/k] \oplus S'.
\end{equation*}  Let $\widehat{\mathbf T'}$ be the lattice defined by the exact sequence
\begin{equation} \label{big T'}
	0\rightarrow \mathbb Z_{P'} \xrightarrow{f'} \mathbb D' \rightarrow \widehat{\mathbf T'}\rightarrow 0,
\end{equation}
where $f'$ sends any $\tau_i\in \Gamma_k/\Gamma_{L'_i}$ to $(\tau\otimes N_K, f'_{\tau_i})$, where 
\begin{equation*} f'_{\tau_i}=\begin{cases} -[L_i:L'_i],&\text{ if  } n\mid m,\\
	(-[L_i:L'_i]N_K,-[L_i:L'_i]S_K), &\text{ if  } n\nmid m.
	\end{cases}
\end{equation*} 
We can define $\widehat T\rightarrow \widehat{\mathbf T'}$  induced by the morphism
$$ \mathbb Z[K/k]\rightarrow \mathbb D', \sigma \longmapsto (-\sum_{i=1}^s N'_i \otimes \sigma,j_\sigma),$$
where $j_\sigma$ is defined in (\ref{def:j-sigma}) and $N'_i=\sum_{\tau'\in \Gamma_k/\Gamma_{L'_i}} \tau'$. We define   $\widehat{\mathbf T'}\rightarrow \widehat{ T'}$  induced by the morphism
$$ \mathbb D'\rightarrow   \mathbb D, (\tau_i\otimes \sigma,a) \longmapsto (\sum_{\tau'_i\in \Gamma_{L'_i}/\Gamma_{L_i}}\tau_i\tau'_i\otimes \sigma,a) \text{ for any } \tau_i\in \Gamma_k/\Gamma_{L_i'} \text{ and } a\in S'.$$
Therefore,  one obtains the following commutative diagram
\begin{equation}\label{diag:T-T'}
	\xymatrix {
		0\ar[r] &\mathbb Z  \ar[r]\ar[d]  &\mathbb \Z[K/k] \ar[r] \ar[d]   &\widehat{ T}   \ar[r]\ar[d] & 0 \\
		0\ar[r] &\mathbb Z_{P'} \ar[r]\ar[d]  &\mathbb D' \ar[r] \ar[d]   &\widehat{\mathbf T'}   \ar[r]\ar[d] & 0 \\	
		0\ar[r]  &\mathbb Z_P  \ar[r] & \mathbb D \ar[r]  &\widehat{T'}  \ar[r]  & 0.
	}
\end{equation}

We can give a more amenable formula for the group $\Ker[\Sha^2_\omega(\widehat{T})\rightarrow \Sha^2_\omega(\widehat{T'})]$, which is the difficult part of the formula (\ref{seq:main-1}).
\begin{lem}\label{lem:injective-1} Suppose that $K/k$ is a Galois field extension and that $P(t)$ is a separable polynomial. Then $\Sha^2_\omega(\widehat{\mathbf T'})\rightarrow \Sha^2_\omega(\widehat{T'})$ is injective. In particular, we have $$\Ker[\Sha^2_\omega(\widehat{T})\rightarrow \Sha^2_\omega(\widehat{T'})]=\Ker[\Sha^2_\omega(\widehat{T})\rightarrow \Sha^2_\omega(\widehat{\mathbf T'})].$$
\end{lem}
\begin{proof}  The last two rows of (\ref{diag:T-T'})
 yield the exact sequence
\begin{equation}\label{seq:T'}
	0 \rightarrow \widehat{\mathbf T'} \rightarrow \widehat{T'} \rightarrow \mathbb U\rightarrow 0,
\end{equation}
where $\Bbb U$ is the Galois lattice defined by the exact sequence
\begin{equation}\label{seq:U}
	0\rightarrow \mathbb Z_{P'}\otimes \widehat T  \rightarrow \mathbb Z_P\otimes \widehat T\rightarrow \mathbb U \rightarrow 0.
\end{equation}
From sequences (\ref{seq:T'}) and (\ref{seq:U}) we derive the following commutative diagram
\begin{equation*}
	\xymatrix {
		0\ar[r] &H^1(k, \mathbb U) \ar[r]\ar@{=}[d] &H^2(k,\widehat{\mathbf T'})\ar[r]\ar[d] &H^2(k,\widehat{ T'}) \ar[d]\\
		0\ar[r] &H^1(k,\mathbb U) \ar[r] &H^2(k,\mathbb Z_{P'}\otimes \widehat T)  \ar[r] &H^2(k,\mathbb Z_P\otimes \widehat T) ,
	}
\end{equation*}
where the exactness of rows follows from the surjectivity of   $H^1(k,\widehat{\mathbf T'})\rightarrow H^1(k,\widehat{ T'})$ and $H^1(k,\mathbb Z_{P'}\otimes \widehat T)  \rightarrow H^1(k,\mathbb Z_P\otimes \widehat T)$. Hence, the injectivity of $\Sha^2_\omega(\widehat{\mathbf T'})\rightarrow \Sha^2_\omega(\widehat{ T'})$ follows from the injectivity of $\Sha^2_\omega(\mathbb Z_{P'}\otimes \widehat T)\rightarrow \Sha^2_\omega(\mathbb Z_P\otimes \widehat T)$.

%the sequence (\ref{seq:T'}) gives the exact sequence $$0\rightarrow H^1(k, \mathbb U) \rightarrow H^2(k,\widehat{\mathbf T'})\rightarrow H^2(k,\widehat{ T'}),$$
%hence the injectivity of $H^2(k,\widehat{T'})\rightarrow H^2(k,\widehat{\mathbf T'})$ follows from $H^1(k, \mathbb U)=0$.
%The sequence (\ref{seq:U}) gives the following exact sequence
%$$H^1(k,\mathbb Z_{P'}\otimes \widehat T)  \xrightarrow g H^1(k,\mathbb Z_P\otimes \widehat T)\rightarrow H^1(k,\mathbb U) \rightarrow H^2(k,\mathbb Z_{P'}\otimes \widehat T)  \rightarrow H^2(k,\mathbb Z_P\otimes \widehat T).$$
% Since $g$ is surjective, we have $H^1(k, \mathbb U)=\Ker[H^2(k,\mathbb Z_{P'}\otimes \widehat T)  \rightarrow  H^2(k,\mathbb Z_P\otimes \widehat T )]$.
%Since
%$H^2(k,\mathbb Z_{P'}\otimes \widehat T)\cong \oplus_i H^2(K/L_i',\mathbb Q/\mathbb Z)$ and $H^2(k,\mathbb Z_P\otimes \widehat T)\cong \oplus_i H^2(K.L_i/L_i,\mathbb Q/\mathbb Z).$
%It is clear that $H^2(K/L_i',\mathbb Q/\mathbb Z) \cong H^2(K.L_i/L_i,\mathbb Q/\mathbb Z)$, hence $H^1(k, \mathbb U)=0$.
%Let $E/k$ be the composite $K^c.L^c$. We can see $\Sha^2_\omega(\widehat
%T)=\Sha^2_\omega(E/k,\widehat T)$ since $\widehat T$ splits by $E$.

Let
$$A:=\Ker[\Sha^2_\omega(\mathbb Z_{P'}\otimes \widehat T)\rightarrow \Sha^2_\omega(\mathbb Z_P\otimes \widehat T)]
= \Ker[\Sha^2_\omega(\mathbb Z_{P'}\otimes \widehat T)\rightarrow
H^2(k,\Z_P\otimes \widehat T)]. $$
For any Galois extension $F'/F$ and $\Gal(F'/F)$-module $M$, we denote
$\Sha^2_\omega(F'/F,M)$  the subgroup of $H^2(F'/F,M)$
consisting of those elements whose image in $H^2(H,M)$ vanishes for
every cyclic subgroup~$H$ of~$\Gal(F'/F)$.
Using Shapiro's lemma,
we have
$$A\cong \oplus_i\Ker[\Res_{L_i'/L_i}:\Sha^2_\omega(\bar k/L_i',\widehat T)\rightarrow
H^2(L_i, \widehat T)].$$ We have the following commutative diagram
\begin{equation}\label{diag:Sha}
	\xymatrix @R=15pt @C=15pt{ &\Sha^2_\omega(\bar k/L_i',\widehat T)\ar[r] &
		H^2(L_i,\widehat T)\\
		&\Sha^2_\omega(K/L_i',\widehat T) \ar[r] \ar[u]^{f} &
		H^2(K.L_i/L_i,\widehat T)\ar[u]^{g}, }
\end{equation}
where $f$ is an
isomorphism since $\widehat T$ splits by $K$, $g$ is injective by Hochschild-Serre spectral sequence and $H^1(K.L_i,\widehat T)=0$.  Since $L_i\cap K=L_i'$,  the second horizontal map of (\ref{diag:Sha})  is injective. Therefore
$$A\cong \oplus_i\Ker[\Sha^2_\omega(K/L_i',\widehat
T) \rightarrow H^2(K.L_i/L_i,\widehat T)]=0. \qedhere$$
\end{proof}

In the remainder of this section, we always assume that \emph{$P(t)$ is an irreducible polynomial, $L=k[t]/(P(t))$ and that the fields $K$  and $L':=L\cap K$ are Galois over $k$}. Let  $G=\Gal(K/k)$, $H=\Gal(K/L')$ and $l=[L:L']$.
\begin{lem} \label{lem:sha-inj} The natural map
$$\Ker[\Sha^2_\omega(\widehat{\mathbf T'})\rightarrow H^2(H,\widehat{\mathbf T'})]\rightarrow H^1(G/H,H^1(H,\widehat{\mathbf T'}))$$
is injective.
\end{lem}

%Since $S'$ is a permutation $G$-module, we have $H^1(H', S')=0$. By  Hochschild-Serre's spectral sequence, the map $H^2(G'/H', S'^{H'})\rightarrow H^2(G', S')$ is injective.
Let $g\in G$, $G'=\left<g\right>$ and  $H'=\left<g\right>\cap H$.  We will prove two lemmas before the proof of Lemma~\ref{lem:sha-inj}.
\begin{lem}\label{lem: inj-H'}
	The morphism $H^2(G'/H', \widehat {\mathbf T'}^{H'})\rightarrow H^2(G', \widehat {\mathbf T'})$ is injective.
\end{lem}
\begin{proof} By Hochschild-Serre's spectral sequence, we only need to show that $H^1(G', \widehat {\mathbf T'}) \rightarrow H^1(H', \widehat {\mathbf T'})^{G'/H'}$ is surjective.
	
	Let $k'=K^{G'}$ and $P(t)=q_1(t)q_2(t)\cdots q_r(t)$ with all $q_i(t)$ irreducible over $k'$. Recall $l=[L:L']$. Let $L_i'=k'[t]/(q_i(t))$ for $1 \leq i \leq r$. All $\Gal(K/K\cap L_i')$ are isomorphic and we denote them by $H'$.

Let $ \# H'=d_1, \#(G'/H')=d_2 d_3$, where $d_3$ is the maximal factor of $\#(G'/H') $ such that $\gcd(d_1,d_3)=1$, hence $\gcd(d_2,d_3)=1$. Recall $e'=\mathrm{gcd}(n,\deg{P(t)})$, hence $d_2d_3\mid e'$. Let $e''=e'/(d_2d_3)$. By (\ref{big T'}), we have
	\begin{align}\notag
	H^1(G', \widehat {\mathbf T'})&\cong  \{(\chi_i)_i\in H^1(H', \Q/\Z)^r: l\sum_{i=1}^r \Cor_{H'/G'}(\chi_i)(g)=0\text{ for any } g\in G',g^{e'}=1_{G'} \}\\
	\notag
	&\cong  \{(\widetilde\chi_i|_{H'})_i:\widetilde\chi_i\in H^1(G', \Q/\Z)^r, l\sum_{i=1}^r d_2d_3\widetilde\chi_i(h)=0\text{ for any } h\in G', h^{e'}=1_{G'} \}.\\
	\label{seq: H1G'}
	&= \{(\widetilde\chi_i|_{H'})_i:\widetilde\chi_i\in H^1(G', \Q/\Z)^r, l\sum_{i=1}^r\widetilde\chi_i(h)=0\text{ for any } h\in H', h^{e''}=1_{H'} \}.
\end{align}
Similarly by (\ref{big T'}), we have
\begin{equation*}
	H^1(H', \widehat {\mathbf T'})\cong \{\sum_{\sigma\in G/H}\chi_\sigma\otimes \sigma \in H^1(H', \Q/\Z)\otimes \Z[G/H]: l\sum_{\sigma\in G/H}\chi_\sigma(h)=0 \text{ for any } h\in H', h^{e'}=1_{H'}\}.
	%\cong \oplus_{i=1}^r H^1(H', \Q/\Z)\otimes \Z[G'/H'].
\end{equation*}
Therefore
\begin{align}\notag
	H^1(H', &\widehat {\mathbf T'})^{G'/H'}\cong  \{(\chi_i)_i\in H^1(H', \Q/\Z)^r: ld_2d_3\sum_{i=1}^r\chi_i(h)=0 \text{ for any } h\in H', h^{e'}=1_{H'}\}\\
	\label{seq: H1H'}
	&=\{(\chi_i)_i\in H^1(H', \Q/\Z)^r: l\sum_{i=1}^r\chi_i(h)=0\text{ for any } h\in H', h^{e''}=1_{H'} \}.
	%\oplus_{i=1}^r H^1(H', \Q/\Z)
\end{align}
By (\ref{seq: H1G'}) and (\ref{seq: H1H'}), the cyclicity of $G'$ implies  that $H^1(G', \widehat {\mathbf T'}) \rightarrow H^1(H', \widehat {\mathbf T'})^{H'}$ is surjective.
\end{proof}

\begin{lem}\label{lem:3-inj} 	$H^2(G/H, \widehat {\mathbf T'}^{H}) \rightarrow \prod_{G' \text{ cyclic}} H^2(G'/H',\widehat {\mathbf T'}^{H'})$ is injective.
\end{lem}
\begin{proof}
By the exact sequence (\ref{big T'}), we have the exact sequence $$0\rightarrow \mathbb Z_{P'} \rightarrow \mathbb D'^H\rightarrow \widehat {\mathbf T'}^{H}\rightarrow 0. $$
Hence
\begin{equation}\label{iso:1}
	H^2(G/H, \widehat {\mathbf T'}^{H}) \cong H^2(G/H, \mathbb D'^H)\cong H^2(G/H, S'^H).
\end{equation}	
Similarly  by the exact sequence (\ref{big T'}), we have
\begin{equation}\label{iso:2}
	H^2(G'/H',\widehat {\mathbf T'}^{H'}) \cong H^2(G'/H', \mathbb D'^{H'})\cong H^2(G'/H', S'^{H'}).	
\end{equation}	
By (\ref{iso:1}) and (\ref{iso:2}), it is enough to show that $H^2(G/H, S'^H)\rightarrow \prod_{G' \text{ cyclic}} H^2(G'/H', S'^{H'})$ is injective.
The injectivity follows from the fact that $S'$ is a $G$-permutation module and that $S'^H$ is a direct factor of $S'^{H'}$ as $G'/H'$-modules.
%
%If $n\mid \deg P(t)$, then $S'=\mathbb Z$. Since $S'^H=S'^{H'}=\mathbb Z$ is a permutation module, hence
%$H^2(G/H, S'^H)\rightarrow  \prod_{G' \text{ cyclic}} H^2(G'/H', S'^{H'})$ is injective.
%
%Suppose $n\nmid \deg P(t)$. Denote $\widetilde{G} =G'/H'$, we can write $S'^{H'}\cong \oplus_{\widetilde H} \mathbb  Z[\widetilde{G}/\widetilde H]^{d_{\widetilde H}}$. If $d_{\widetilde H} >0$, such $\mathbb Z[\widetilde{G}/\widetilde H]$ must appear in $S'^H$ as a $G'/H'$-module. Hence, the two maps $H^2(G/H, S'^H) \rightarrow H^2(G'/H', S'^H)$
%	and  $H^2(G/H, S'^H) \rightarrow H^2(G'/H', S'^{H'})$  are  the same kenerl. Note that $S'^H$ is a permutation $G/H$-module, hence $H^2(G/H, S'^H) \rightarrow \prod_{G' \text{ cyclic}} H^2(G'/H', S'^H)$ is injective. Therefore,  the map $H^2(G/H, S'^H) \rightarrow \prod_{G' \text{ cyclic}} H^2(G'/H', S'^{H'})$ is injective.
%%	 hence $H^2(G/H, \widehat {\mathbf T'}^{H}) \rightarrow \prod_{G' \text{ cyclic}} H^2(G'/H',\widehat {\mathbf T'}^{H'})$ is injective.
\end{proof}	
\begin{proof}[The proof of Lemma \ref{lem:sha-inj}:]
	By Hochschild-Serre's spectral sequence, we have the commutative diagram of exact sequences
\begin{equation*}
	\xymatrix {
		H^2(G/H, \widehat {\mathbf T'}^H)\ar[r]\ar[d]^{f_1} &\Ker[H^2(G, \widehat {\mathbf T'}) \rightarrow H^2(H, \widehat {\mathbf T'})]\ar[r]^{\psi}\ar[d]^{f_2}  &H^1(G/H,H^1(H,\widehat {\mathbf T'}))\\	
	\prod_{g} H^2(G'/H', \widehat {\mathbf T'}^{H'})\ar[r] &\Ker[\prod_{g} H^2(G', \widehat {\mathbf T'}) \rightarrow \prod_{g} H^2(H', \widehat {\mathbf T'})],
	}
\end{equation*}
where $g$ runs through every element in $G$, $G'=<g>$ and $H'=<g>\cap H$.

Suppose $\alpha \in \Ker{f_2}\cap \Ker{\psi}$, then there exists an element $\beta\in H^2(G/H, \widehat {\mathbf T'}^H)$ such that $\alpha$ is the image of $\beta$.  The map in the second row is injective by Lemma \ref{lem: inj-H'}, hence $f_1(\alpha)=0$. Since $f_1$ is injective by Lemma \ref{lem:3-inj}, we have $\alpha=0$. Then we complete the proof of Lemma \ref{lem:sha-inj}.
\end{proof}

Let $G$ be a finite group and $H$ a normal subgroup of $G$. Let $[H,H]$ be the commutator subgroup of $H$ and $H^{ab}=H/[H,H]$. The group extension
\begin{equation}\label{extension}
1\rightarrow H^{ab} \rightarrow G/[H,H]\rightarrow G/H \rightarrow 1
\end{equation}
defines a cohomology class $u\in H^2(G/H, H^{ab})$. Let $A$ be a $G$-module with trivial $H$-action. For any
$p>0$, the cup product defines a morphism
$$u\cup : H^{p-1}(G/H,H^1(H,A))\rightarrow H^{p+1}(G/H, A).$$
For any $x\in H^{p-1}(G/H,H^1(H,A))$, we have $u\cup x=-d_2^{p-1,1}(x)$ by \cite[Theorem 2.1.8]{NSW}, where $d_2^{p-1,1}$ is the differential in Hochschild-Serre's spectral sequence.

Let $P(t)$ be an irreducible polynomial. Let $L=k[t]/(P(t))$. Suppose that the fields $K$  and $L'=L\cap K$ are Galois over $k$. Let $G=\Gal(K/k)$, $H=\Gal(K/L')$. Recall that $e'=\text{gcd}(n,\deg P(t))$ and $l=[L:L']$.

%Suppose $e'\neq n$, $i.e.$, $n\nmid \deg P(t)$.
%Then $S'=\mathbb Z[\Omega]$, where $\Omega=\{M: M\subset G, |M|=e'\}$, see (\ref{D}).

Denote \begin{align}\label{def:C}
	C:&=\{\chi\in H^1(H,\Bbb Q/\Bbb Z): l\Cor_{L'/k}(\chi)(g)=0 \text{ for any }g\in G \text{ with } g^{e'}=1_G\},\\
\label{def:C'}
C':&=\{\sum_{\sigma\in G/H}\chi_\sigma\otimes\sigma\in H^1(H,\Bbb Q/\Bbb Z)\otimes \mathbb Z[G/H]: l\sum_\sigma\chi_{\sigma}(h)=0 \text{ for any }h\in H \text{ with } h^{e'}=1_H\}.
\end{align}

%Suppose $e'=n$, $i.e.$, $n\mid \deg P(t)$.
%Then $S'=\mathbb Z$.
%Denote \begin{align}\label{def:C2}
%	C:&=\{\chi\in H^1(H,\Bbb Q/\Bbb Z): [L:L']\Cor_{L'/k}(\chi)=0 \},\\
%	\label{def:C'2}
%	C':&=\{\sum_{\sigma\in G/H}\chi_\sigma\otimes\sigma \in H^1(H,\Bbb Q/\Bbb Z)\otimes \mathbb Z[G/H]: \sum_\sigma \chi_\sigma=0 \}.
%\end{align}
% We have
%\begin{align*}H^1(k,\widehat{T'})&\cong \Ker[H^2(L,\mathbb Z)\rightarrow  H^2(L\otimes K,\mathbb Z)\oplus H^2(k,\mathbb Z)]\\
%	&=\Ker[H^2(K.L/L,\mathbb Z) \rightarrow H^2(k,\mathbb Z)]\cong C,\\
%H^1(H,\widehat{\mathbf  T'})&\cong \Ker[H^2(H,\mathbb Z)\otimes \mathbb Z[G/H]\rightarrow H^2(k,\mathbb Z)]=C'.
%\end{align*}

\begin{thm}\label{prop:Br} Let $P(t)$ be an irreducible polynomial and $L=k[t]/(P(t))$. Suppose that the fields $K$  and $L'=L\cap K$ are Galois over $k$. Let $X$ be the variety (\ref{def:X}) defined by  the equation (\ref{eq:norm}). Let $G=\Gal(K/k)$ and $H=\Gal(K/L')$.  Then we have the exact sequences
\small\begin{align}\label{seq:Br}
		 0&\rightarrow H^1(G/H,\Bbb Q/\Bbb Z)\rightarrow H^1(G,\Bbb Q/\Bbb Z)\rightarrow C \rightarrow \Br_\mathrm{un}(X)/\Br_0(X)  \rightarrow \Ker[\Sha^2_\omega(\widehat{T})\rightarrow \Sha^2_\omega(\widehat{T'})]\rightarrow 0,\\
    \notag 
	0&\rightarrow H^1(G/H,\Bbb Q/\Bbb Z)\rightarrow H^1(G,\Bbb Q/\Bbb Z)\rightarrow H^1(H,\Bbb Q/\Bbb Z)^{G/H}\rightarrow H^2(G/H,\mathbb Q/\mathbb Z) \rightarrow \Ker[\Sha^2_\omega(\widehat{T}) \\
	\label{seq:Br2}
&\rightarrow \Sha^2_\omega(\widehat{T'})] \rightarrow H^1(G/H, H^1(H,\mathbb Q/\mathbb Z))
\xrightarrow{(\Res,u\cup)}  H^1(G/H,C')\oplus H^3(G/H,\mathbb Q/\mathbb Z),
\end{align}
where $C$ and $C'$ are defined by (\ref{def:C}) and (\ref{def:C'}), and $\Res$ is induced by the diagonal map $H^1(H,\mathbb Q/\mathbb Z)\to C'$.
\end{thm}
\begin{proof}
	Suppose $e'\neq n$, $i.e.$, $n\nmid \deg P(t)$. Then $S'=\Z[K/k]\oplus\mathbb Z[\Omega]$, where $\Omega=\{M: M\subset G, |M|=e'\}$, see (\ref{D}). As a $G$-module, $\Z[\Omega]$ is the direct sum of submodules which are isomorphic to $\Bbb Z[G/E]$ where $E$ is a subgroup of $G$ with degree dividing $e'$; arbitrary such submodule $\Bbb Z[G/E]$ must occur in $\Z[\Omega]$. Hence, by (\ref{seq:new T'}) and (\ref{big T'}), we have
	\begin{align}\notag
		H^1(k,\widehat{T'})&\cong \Ker[H^2(L,\mathbb Z)\rightarrow  H^2(L\otimes K,\mathbb Z)\oplus H^2(k,S')]\\
		\label{iso:T'}
		&=\Ker[H^2(K.L/L,\mathbb Z) \rightarrow H^2(k,\Z[\Omega])] \cong C, \\
		\label{iso:HT'}
		H^1(H,\widehat{\mathbf T'})&\cong \Ker[f'^*: H^2(H,\mathbb Z)\otimes \mathbb Z[G/H]\rightarrow   H^2(H,\Z[\Omega])]\cong C',
	\end{align}
where $f'^*$ is $-l$ times the natural map.
	It is clear that (\ref{iso:T'}) and (\ref{iso:HT'}) also hold when $e'=n$, i.e., $n\mid \deg{P(t)}$. On the other hand, we have $H^1(k,\widehat{T})\cong  H^1(G,\Bbb Q/\Bbb Z)$ and $H^1(H,\widehat{T})\cong  H^1(H,\Bbb Q/\Bbb Z)$. The first exact sequence then follows from (\ref{seq:main-1}).

	By Hochschild-Serre's spectral sequence, we have the commutative diagram with exact sequences
\small \begin{equation}\label{diagram:H2T}
	\xymatrix {
		H^2(G/H, \widehat {T}^H)\ar[r]\ar[d] &\Ker[H^2(G, \widehat {T}) \ar[r]^{f_1} \ar[d]_{f_2} &H^2(H, \widehat {T})]\ar[r]\ar[d]  &H^1(G/H,H^1(H,\widehat {T}))\ar[r]\ar[d] &H^3(G/H,\widehat {T}^H)\\
				H^2(G/H, \widehat {\mathbf T'}^H)\ar[r]  &\Ker[H^2(G, \widehat {\mathbf T'}) \ar[r] &H^2(H, \widehat {\mathbf T'})]\ar[r]  &H^1(G/H,H^1(H,\widehat {\mathbf T'})).\\	
	}
\end{equation}
 The exact sequence (\ref{big T'}) induces a morphism $\widehat {\mathbf T'}\rightarrow \widehat T\otimes \Bbb Z_{P'}$ by  sending $S'$ to $0$. Note that $H^2(G,\widehat T\otimes \Bbb Z_{P'})\cong H^2(H,\widehat T)$, hence $f_1$ factors through $f_2$,  so $\Ker(f_2)\subset \Ker(f_1)$. By the exact sequence
\begin{equation}\label{seq:T}
	0\rightarrow \mathbb Z \rightarrow \mathbb Z[G]\rightarrow \widehat T\rightarrow 0,
\end{equation}
one obtains $\Sha^2_\omega(\widehat{T})=H^2(G,\widehat{T})$.  Therefore, by Lemma \ref{lem:sha-inj}, the diagram (\ref{diagram:H2T}) yields the commutative diagram
\begin{equation}\label{diag:HT2}
	\xymatrix {
		H^2(G/H, \widehat {T}^H)\ar[r]\ar[d] &\Ker[\Sha^2_\omega(\widehat{T}) \ar[r]^{f_1} \ar[d]_{f_2} &H^2(H, \widehat {T})]\ar[r]\ar[d]  &H^1(G/H,H^1(H,\widehat {T}))\ar[r]\ar[d] &H^3(G/H,\widehat {T}^H)\\
		0\ar[r]  &\Ker[\Sha^2_\omega(\widehat {\mathbf T'}) \ar[r] &H^2(H, \widehat {\mathbf T'})]\ar[r]  &H^1(G/H,H^1(H,\widehat {\mathbf T'})),	
	}
\end{equation}
Let $A=(\Bbb Q\oplus \Bbb Z[G])/\Bbb Z$, then we have the morphisms of spectral sequences
\begin{equation}\label{ss:1}
	E(\Bbb Q/\Bbb Z)\leftarrow E(A) \rightarrow E(\widehat{T}).
\end{equation}
By (\ref{ss:1}) and the exact sequence (\ref{seq:T}), we have the morphism of spectral sequences $ E(\widehat{T})\rightarrow E(\Bbb Q/\Bbb Z) $. By the exact sequence (\ref{big T'}), the diagram (\ref{diag:HT2}) is equivalent to the following commutative diagram
\begin{equation} \label{diag: Sha-T}
	\xymatrix {
		H^2(G/H, \mathbb Q/\mathbb Z)\ar[r]\ar[d] &\Ker[\Sha^2_\omega(\widehat{T}) \ar[r]^{f_1} \ar[d]_{f_2} &H^2(H, \Bbb Q/\Bbb Z)]\ar[r]\ar[d]  &H^1(G/H,H^1(H,\mathbb Q/\mathbb Z))\ar[r]^{u\cup}\ar[d]^{\Res} &H^3(G/H,\mathbb Q/\mathbb Z)\\
		0\ar[r]  &\Ker[\Sha^2_\omega(\widehat {\mathbf T'}) \ar[r] &H^2(H, \widehat {\mathbf T'})]\ar[r]  &H^1(G/H,C').
	}
\end{equation}
By Lemma \ref{lem:injective-1}, $\Ker[\Sha^2_\omega(\widehat{T})\rightarrow \Sha^2_\omega(\widehat{T'})]=\Ker(f_2)\subset \Ker(f_1)$, the proof then follows from the commutativity of the diagram (\ref{diag: Sha-T}).
\end{proof}

\begin{rem} If $L\cap K=k$, one obtains  $\Br_\mathrm{un}(X)=\Br_0(X)$ by Theorem \ref{prop:Br} and this result has been already proved in \cite[Theorem 3.2]{Wei14}.
\end{rem}

For any finite group $M$, we denote  $[M,M]$ the commutator subgroup  of $M$ and $M^{ab}:=M/[M,M]$. 
\begin{cor} Let $X,G$ and $H$ be defined as in Theorem \ref{prop:Br}. Suppose $H^{ab}=0$. Then we have $\Br_\mathrm{un}(X)/\Br_0(X) \cong H^3(G/H,\Bbb Z)$. Furthermore, if we assume that $G/H$ is cyclic, then $\Br_\mathrm{un}(X)=\Br_0(X)$.
\end{cor}
\begin{proof} Since $H^{ab}=0$, we have $C=0$ and $H^1(H,\Bbb Q/\Bbb Z)=0$. By Theorem \ref{prop:Br}, we have $$\Br_\mathrm{un}(X)/\Br_0(X) \cong H^2(G/H,\Bbb Q/\Bbb Z)\cong H^3(G/H,\Bbb Z). $$
If $G/H$ is cyclic, then $H^3(G/H,\Bbb Z)\cong H^1(G/H,\Bbb Z)=0. $
\end{proof}

\begin{lem} \label{cor:abel}
	\begin{enumerate}[1)]
	\item Suppose that $G/[H,H]$ is abelian. Then $\Br_\mathrm{un}(X)/\Br_0(X)$ is  isomorphic to the subgroup $B$ of $H^2(G, \Q/\Z)$ and the following exact sequence holds
	\begin{align}\notag
		0\rightarrow H^2(G/H, \Q/\Z) \rightarrow &B \rightarrow H^1(G/H, H^1(H,\Q/\Z)) \\
		\label{seq:abel}
		&\xrightarrow{(\Res, u\cup)}H^1(G/H,C')\oplus H^3(G/H,\mathbb Q/\mathbb Z);
	\end{align}

\item Let $P(t)$ be an irreducible polynomial and $L=k[t]/(P(t))$. Let $L'\subset L$ be a Galois extension of $k$. Let $F/k$ be a Galois field extension and $F\cap L=k$.  Suppose $K=L'. F$.
Then $$\Br_\mathrm{un}(X)/\Br_0(X)\cong  H^2(L'/k, \Q/\Z)\oplus \Delta,$$ where $\Delta$ is the kernel of the map $\Res: H^1(L'/k, H^1(F/k,\Q/\Z)) \rightarrow H^1(L'/k,C')$.
	\end{enumerate}
\end{lem}
\begin{proof} For the case 2), we have $G=\Gal(K/k) \cong \Gal(F/k)\times \Gal(L'/k)$, $H=\Gal(K/L')\cong \Gal(F/k)$, $G/H\cong \Gal(L'/k)$ and  $K\cap L=L'$.
	In both cases, the two maps $H^1(G,\Bbb Q/\Bbb Z)\rightarrow C$ in (\ref{seq:Br}) and $H^1(G,\Bbb Q/\Bbb Z)\rightarrow H^1(H,\Bbb Q/\Bbb Z)$
 in (\ref{seq:Br2}) are surjective, hence the exact sequence (\ref{seq:abel}) holds for both cases by Theorem \ref{prop:Br}. Now it remains to prove the case 2). Since $\Gal(K/k) \cong \Gal(F/k)\times \Gal(L'/k)$, the cohomology class $u=0$ since the sequence (\ref{extension}) is split and the third map in (\ref{seq:abel}) is split by the main theorem of~\cite{Jan90}. The proof for the case 2) then follows.
\end{proof}
%Let $C''=\{(\chi_i)\in H^1(H,\Bbb Q/\Bbb Z)\otimes \mathbb Z[G/H]: \sum_i \chi_i=0 \}$, which is a submodule of $C'$. Obviously, we have the following exact sequence
%$$0\rightarrow C''\rightarrow H^1(H,\Bbb Q/\Bbb Z)\otimes \mathbb Z[G/H]\rightarrow H^1(H,\Bbb Q/\Bbb Z) \rightarrow 0,$$
%it gives $H^1(G/H,C'') \cong H^0(G/H, H^1(H,\Bbb Q/\Bbb Z))$.

\begin{lem}\label{lem:map-0}
Suppose $G/H$ trivially acts on $H^1(H,\Bbb Q/\Bbb Z)$ ($e.g.$, $G/[H,H]$ is abelian). Let $\chi\in H^1(G/H,H^1(H,\Bbb Q/\Bbb Z))$ have order $r$. Then 
\begin{equation}\label{eq-zero}
	2\Res(\chi)=0 \in H^1(G/H,C').
\end{equation} 
Furthermore, if $r$ is odd or $2\mid l\#\Ker(\chi_h)$ for any $h\in H$ with $h^{e'}=1_{H}$, then $$\Res(\chi)=0,$$
where $l=[L:L']$ and $\chi_h: G/H\to \Q/\Z$ by sending $\sigma\in G/H$ to $\chi(\sigma)(h)$.
\end{lem}
\begin{proof}
Let $a:=-\sum_{\sigma\in G/H}\chi(\sigma)\otimes\sigma\in H^1(H,\Bbb Q/\Bbb Z)\otimes \mathbb Z[G/H]$. For any $\tau \in G/H$, then
\begin{align} \notag
	\partial^1(a)(\tau):&=\tau a-a=-\sum_{\sigma\in G/H}\chi(\sigma)\otimes\tau\sigma+\sum_{\sigma\in G/H}\chi(\sigma)\otimes\sigma\\
	\label{compute:a}
	&=-\sum_{\sigma\in G/H}\chi(\tau^{-1}\sigma)\otimes\sigma+\sum_{\sigma\in G/H}\chi(\sigma)\otimes\sigma=\chi(\tau)\otimes \sum_{\sigma\in G/H}\sigma.
\end{align}
Therefore $\partial^1(a)=\Res(\chi)\in C^1(G/H,H^1(H,\Bbb Q/\Bbb Z)\otimes \mathbb Z[G/H])$. To complete the proof, we only need to show $a\in C'$. For any $h\in H$, one has \begin{equation}\label{sum=0}
	l\sum_{\sigma\in G/H} \chi(\sigma)(h)=l\#\Ker(\chi_h)r'(r'-1)/2\cdot \psi \in  \Bbb Q/\Bbb Z,\end{equation}
where $\psi$ is a generator of the image of $\chi_h$ in $\Q/\Z$ which has order $r'$. Obviously $2a\in C'$, hence we proved (\ref{eq-zero}).  Since $r'\mid r$,
if $r$ is odd or $2\mid [L:L']\#\Ker(\chi_h)$, then the sum in (\ref{sum=0}) is zero, hence $a\in C'$, the proof then follows.
\end{proof}

 Let $V$ be the maximal smooth open subset of (\ref{eq:norm}).  Let $V^{\mathrm{CTHS}}$ be the
CTHS-partial compactification of $V$ (see  \S 2 or \cite{CHS}). By Theorem \ref{prop:Br}, we immediately obtain the following result which is slightly more general than \cite[Theorem 3.6]{Wei14}.
\begin{cor}\label{wei14}  Suppose $G/[H,H]$ is an abelian group. Then:
	\begin{enumerate}[(a)]
		\item The quotient $\Br(V^{\mathrm{CTHS}})/\Br_\mathrm{un}(X)$ is $2$-torsion.
		
		\item $\Br_\mathrm{un}(X)=\Br(V^{\mathrm{CTHS}})$ if one of the following conditions holds:
		\begin{enumerate}[(1)]
			\item either $H^{ab}$  or $G/H$ has odd order, or $G \cong \Z/2^i \times A$, where $A$ has odd order;
			
			\item $[L:L\cap K]$ is even;
			
			\item there exists $s\geq 1$ such that $2^{s}\mid \#(G/H)$ and $\# H^{ab}=2^{s-1}d$, where $d$ is odd;
			
			\item $L/k$ contains an abelian subfield $L''/k$ with
			$\Gal(L''/k)\cong(\Bbb Z/2\Bbb Z)^2$.
		\end{enumerate}
	\end{enumerate}
\end{cor}
\begin{proof}
	Since $G/[H,H]$ is an abelian group, by \cite[Proposition 2.5]{CHS}, one has
	\begin{equation}\label{seq:CTHS}
		\Br(V^{\mathrm{CTHS}})/\Br_0(V^{\mathrm{CTHS}})\cong \Ker[\Sha^2_\omega(\widehat{T})\rightarrow H^2(H, \widehat{T})]\cong \Ker[H^2(G,\Q/\Z)\rightarrow H^2(H,\Q/\Z)].
	\end{equation}
	Since $G/[H,H]$ is an abelian group, $ H^1(G,\Q/\Z)\rightarrow C$ is surjective. By Theorem (\ref{prop:Br}),  the case (a) follows from Lemma \ref{lem:map-0}. If $\Res:  H^1(G/H, H^1(H,\mathbb Q/\mathbb Z))
	\rightarrow H^1(G/H,C')$ is a zero map, comparing  Theorem \ref{prop:Br} with (\ref{seq:CTHS}),   we have $\Br_\mathrm{un}(X)=\Br(V^{\mathrm{CTHS}})$.
	
	Let $L'=L\cap K$ and $l=[L:L']$. If $G \cong \Z/2^i \times A$ and $A$ has odd order, then $H^2(G,\Q/\Z) \cong H^2(A, \Q/\Z)$ has odd order, hence the image of $B\to H^1(G/H,H^1(H,\Q/\Z))$ in (\ref{seq:abel}) has odd order.
	In the following we will apply Lemma \ref{lem:map-0}, it is clear that $r$ is odd for the case (b-1), $2\mid l$ for the case (b-2) and $2\mid \#\Ker(\chi_h)$ for the case (b-3), the proof then follows from Lemma \ref{lem:map-0}. For the case (b-4), if $L''\not \subset L'$, then $2\mid[L:L']$, the proof follows from the case (b-2), hence we may assume $L''\subset L'$, then one obtains $2\mid \#\Ker(\chi_h)$,  the proof then follows from Lemma~\ref{lem:map-0}.
	%Note that any element in $H^1(G/H,H^1(H,\Bbb Q/\Bbb Z))$ is a sum of elements whose images are cyclic.
\end{proof}
\begin{rem}
	We can compare this corollary with \cite[Theorem 3.6]{Wei14}. Replacing the  assumption that $G$ is abelian in \cite[Theorem 3.6]{Wei14}  we only need to assume that $G/[H,H]$ is abelian. In the case case (b-1) we obtains the new case that $H^{ab}$ is odd order, and the case (b-4) is slightly more general than the case (b-5) of \cite[Theorem 3.6]{Wei14} where the assumption needs $\Gal(L''/k)\cong(\Bbb Z/2\Bbb Z)^3$.
\end{rem}

The following result gives an explicit description for the case 2) of Lemma \ref{cor:abel}.
\begin{prop}\label{prop:split} Let $P(t)$ be an irreducible polynomial and $L=k[t]/(P(t))$. Let $L'\subset L$ be a Galois extension over $k$. Let $F/k$ be a Galois field extension and $F\cap L=k$.  Let $K=L'. F$ and $H=\Gal(K/L')\cong \Gal(F/k)$.	
	 %We write $H_2=\prod_{i=1}^u\Z/2^{s_i}\times \prod_{j=1}^v\times\Z/2^{t_j}$ with $s_i\leq s$ and $t_j <s$ for any $i,j$.
	 Let $X$ be the variety (\ref{def:X}) defined by the equation (\ref{eq:norm}).	Let $l= [L:L']$ and $\rho=\#[\Gal(L'/k),\Gal(L'/k)]$. We write $$
	 		H^{ab}=H_2\times H_{odd}, \Gal(L'/k)^{ab}=H'_2\times H'_{odd},
	 	$$ where  $H_2 \text{ and } H'_2$ have $2$-power order, $H_{odd}\text{ and } H'_{odd}$ have odd order.
	 
	 Suppose  that either $2\mid l\rho$ or  $H'_2$ is zero or non-cyclic. Then $$\Br_\mathrm{un}(X)/\Br_0(X)\cong H^2(L'/k,\Q/\Z)\times H^1(L'/k,H^1(F/k,\Q/\Z)).$$
	 	
	Suppose  $2\nmid l\rho$ and $H'_2\cong \Z/2^{s}$ with $s\geq 1$. 	Let $\widetilde H_{2}\times E$ be the image of $\{h\in H: h^{2^s}=1_H\}$ by the projection $H\rightarrow H^{ab}$, where $\widetilde H_{2}\cong (\Z/2^s)^{u_1}$ and $E$ is killed by $2^{s-1}$. We write $$H_2=\widetilde H_{2}\times (\Z/2^s)^{u_2} \times\prod_{i=1}^v\Z/2^{s_i}\times \prod_{j=1}^w\Z/2^{t_j}$$ with $s_i<s$  and $t_j >s$ for any $i,j$.
	Then $$\Br_\mathrm{un}(X)/\Br_0(X)\cong H^2(L'/k,\Q/\Z)\times H^1(H'_{odd},H^1(H_{odd},\Q/\Z))\times (\Z/2^{s-1})^{u_1}\times (\Z/2^{s})^{w+u_2}\times \prod_{i=1}^v\Z/2^{s_i}.$$

\end{prop}
\begin{proof} By the case 2) of Lemma \ref{cor:abel}, we only need to compute the kernel $\Delta$ of the map $$\Res: H^1(L'/k, H^1(F/k,\Q/\Z)) \rightarrow H^1(L'/k,C').$$
	
	If $2\mid l$, then $\Res$ is a zero map by Lemma \ref{lem:map-0}. If  either $2 \mid \rho$ or $H'_2$ contains a subgroup $\Z/2\times \Z/2$, then
	$2\mid \#\Ker(\chi_h)$ for any $h\in H$;
	if $H'_2=0$, then $r$ is odd; the map $\Res$ is zero for all above cases by Lemma \ref{lem:map-0}. The proof of the case 1) then follows from the case 2) of Lemma \ref{cor:abel}.
	
	Now it remains the case that $2\nmid l\rho$ and $H'_2\cong \Z/2^{s}$ with $s\geq 1$. It is clear that $$H^1(L'/k, H^1(F/k,\Q/\Z))=H^1(\Z/2^s, H^1(H_2,\Q/\Z))\times H^1(H'_{odd}, H^1(H_{odd},\Q/\Z)).$$
Obviously $\Res$ restricting to $H^1(H'_{odd}, H^1(H_{odd},\Q/\Z))$ is zero by Lemma \ref{lem:map-0}. So it remains to compute the kernel of $\Res\mid_{H^1(\Z/2^s, H^1(H_2,\Q/\Z))}$.
We can write $$H^1(\Z/2^s, H^1(H_2,\Q/\Z))=A_1\times A_2\times A_3\times A_4, $$
where \begin{align} \notag
	&A_1:=H^1(\Z/2^s, H^1(\widetilde H_{2},\Q/\Z)),A_2:=H^1(\Z/2^s, H^1(\Z/2^s,\Q/\Z))^{u_2},\\
	\label{def:A}
	 &A_3:=\prod_{i=1}^vH^1(\Z/2^s, H^1(\Z/2^{s_i},\Q/\Z)),A_4:=\prod_{j=1}^w H^1(\Z/2^s, H^1(\Z/2^{t_j},\Q/\Z)).
	\end{align}
Obviously $A_1\cong (\Z/2^s)^{u_1},A_2\cong (\Z/2^s)^{u_2},A_3\cong \prod_{i=1}^v\Z/2^{s_i}$ and $A_4\cong (\Z/2^s)^w.$

	If $\chi \in A_3$, obviously  $2\mid \#\Ker(\chi_h)$ for any $h\in H$ since $s_i<s$, one obtains $A_3\subset \Delta$ by Lemma \ref{lem:map-0}. In the following we will show that $A_4\subset \Delta$. It suffices that $\Res$ restricting to $H^1(\Z/2^s, H^1(\Z/2^{t_j},\Q/\Z))$ is zero. Let $\chi \in H^1(\Z/2^s, H^1(\Z/2^{t_j},\Q/\Z))$. By Lemma \ref{lem:map-0}, we only need to show $2\mid \#\Ker(\chi_h)$ for any $h\in H $ with $h^{e'}=1_{H}$. Let $\bar h$ be the image of $h$ in $H^{ab}$.  
	Since $2\nmid l$, one obtains $v_2(e')=s$. Since $\Im(\chi)\subset H^1(\Z/2^{t_j},\Q/\Z)$, it suffices to show $2\mid \#\Ker(\chi_h)$ when $\bar h\in \Z/2^{t_j}$ satisfies $\bar h^{2^s}=1_{H^{ab}}$.  Let $N=\max\{2s-t_j,0\}<s$ since $t_j>s$. Let $\sigma$ be a generator of $H'_2$ and $\tau$ a generator of $\Z/2^{t_j}$. We can write $\bar h=\tau^{2^{\alpha}}$ with $\alpha\geq t_j-s$. Therefore $$\chi(\sigma^{2^N})(\bar h)=2^{\alpha+N}\chi(\sigma)(\tau)=0$$
	since $\alpha+N\geq t_j-s +2s-t_j=s$ and the order of $\chi$ is a factor of $2^s$. One has $\sigma^{2^N}\in \Ker(\chi_h)$, it implies $2\mid \#\Ker(\chi_h)$ since $N<s$. 
	
	We may show that $A_2\subset \Delta$ by a similar argument as for $A_4$.  Let $\chi \in H^1(\Z/2^s, H^1(\Z/2^s,\Q/\Z))$. By Lemma \ref{lem:map-0}, we only need to show $2\mid \#\Ker(\chi_h)$ for any $h\in \Z/2^{s}\subset H$ with $h^{e'}=1_{H}$. By our assumption, the generator $\tau$  of $\Z/2^s$ is not contained in the image of $\{h\in H: h^{2^s}=1_H\}$. Therefore, it suffices to show $2\mid \#\Ker(\chi_h)$ for any $\bar h\in \Z/2^{s}\subset H^{ab}$ with $\bar h^{2^{s'}}=1_{H^{ab}}$, where  $s'<s$ is a non-negative integer. Let $\tau$ be a generator of $\Z/2^s\subset H$. We can write $\bar h=\tau^{2^{\alpha}}$ with $\alpha\geq s-s'$. Therefore $$\chi(\sigma^{2^{s'}})(h)=2^{\alpha+s'}\chi(\sigma)(\tau)=0$$
	since $\alpha+s'\geq s-s'+s'=s$ and the order of $\chi$ is a factor of $2^s$. Therefore $\sigma^{2^{s'}}\in \Ker(\chi_h)$, it implies $2\mid \#\Ker(\chi_h)$ since $s'<s$.
	
	For any $\chi\in A_1$, one has $2\chi\in \Delta$ by Lemma \ref{lem:map-0}, it remains to show that $\chi \not \in \Delta$ for any $\chi$ has order $2^s$. Let $$\Theta:=\{\sum_{\sigma\in H'_2}\tau_\sigma\otimes \sigma \in H^1(\widetilde H_{2},\Q/\Z)\otimes\mathbb Z[H'_2]: \sum_{\sigma\in H'_2}\tau_\sigma=0\}.$$ 
	The module $\Theta$ is a direct factor of $C'$ (see (\ref{def:C'})) since $\widetilde H_{2}$ is contained in the image of  $\{h\in H: h^{2^s}=1_H\}$ in $H^{ab}$ and $v_2(e')=s$, hence we only need to show $\Res(\chi)\neq 0 \in H^1(H'_2,\Theta)$.
	
	From the exact sequence of $H'_2$-modules $$0\rightarrow \Theta \rightarrow H^1(\widetilde H_{2},\Q/\Z)\otimes\mathbb Z[H'_2]\xrightarrow{j} H^1(\widetilde H_{2},\Q/\Z) \rightarrow 0,$$
	one derives the isomorphism $\partial: H^1(\widetilde H_{2},\Q/\Z) \xrightarrow{\cong} H^1(H'_2,\Theta)$, where $j$ sends
	$\sum_{\sigma}\tau_\sigma\otimes \sigma$ to $\sum_{\sigma}\tau_\sigma$.
	
	Let $a:=-\sum_{\sigma\in H'_2}\chi(\sigma)\otimes\sigma\in H^1(\widetilde H_{2},\Q/\Z)\otimes \mathbb Z[H'_2]$. Let $\gamma$ be the generator of $ H'_2$ which has order $2^s$. Let $b:=j(a)=-\sum_{\sigma\in H'_2}\chi(\sigma)=2^{s-1}(2^s-1)\chi(\gamma)\neq 0\in H^1(\widetilde H_{2},\Q/\Z)$
	since $\chi$ has order $2^s$. Therefore
	$\partial(b)\neq 0 \in H^1(H'_2,\Theta)$ since $\partial$ is an isomorphism. On the other hand,
	by a similar computation as in (\ref{compute:a}), one has $\partial(b)= \Res(\chi)$, then $\Res(\chi)\neq 0$ and  the proof follows.
\end{proof}

%Let $p$ be a prime number. For any finite abelian group $A$, we write $A\{p\}$ for its $p$-primary part and $A_{odd}$ the direct sum of all $p$-primary part for $p\neq 2$.	
\begin{prop}
Let $P(t)$ be an irreducible polynomial and $p$ a prime number. Let $L=k[t]/(P(t))$ and $L'=L\cap K$. Suppose that $L'/k$ is cyclic of order $p^s$ and that $G=\Gal(K/k)$ is an abelian group. Let $X$ be the variety (\ref{def:X}) defined by the equation (\ref{eq:norm}). We can write $G=G_1 \times G_2$ satisfying $L'\subset K^{G_2}$, $G_1=\Z/p^{s'},s'\geq s$ and $$G_2=(\Z/p^s)^r\times\prod_{i=1}^{r_1} \Z/p^{e_i}\times\prod_{j=1}^{r_2} \Z/p^{\mu_j}\times H_1,$$ where $e_i<s$ and $\mu_j>s$ for any $i, j$ and $H_1$ has order prime to $p$.
\begin{enumerate}[(1)]
	\item If either $p$ is odd or $p=2$ and $2\mid[L:L']$, then $$\Br_\mathrm{un}(X)/\Br_0(X)\cong (\Z/p^s)^{r+r_2}\times\prod_{i=1}^{r_1} \Z/p^{e_i}.$$
	
	\item If $p=2$ and $2\nmid[L:L']$, then $$\Br_\mathrm{un}(X)/\Br_0(X)\cong (\Z/2^{s-1})^r\times (\Z/2^{s})^{r_2}\times\prod_{i=1}^{r_1} \Z/2^{e_i}.$$
\end{enumerate}
\end{prop}
\begin{proof} Let $u'\in H^2(\Bbb Z/p^{s}, \Bbb Z/p^{s'-s})$ be defined by the group extension
	\begin{equation*}
		0\rightarrow \Bbb Z/p^{s'-s} \rightarrow G_1\rightarrow \Bbb Z/p^s \rightarrow 0.
	\end{equation*}
Let $u\in H^2(\Bbb Z/p^s,H^{ab})$ be defined by (\ref{extension}), then
$$u=(u',0)\in H^2(\Bbb Z/p^s,H^{ab})\cong H^2(\Bbb Z/p^s,\Bbb Z/p^{s'-s})\oplus H^2(\Bbb Z/p^s,G_2).$$
 Since $H^2(G_1,\mathbb Q/\mathbb Z)=0$, by Hochschild-Serre's spectral sequence, it implies that the map $$u'\cup: H^1(\Bbb Z/p^s, H^1(\Bbb Z/p^{s'-s},\mathbb Q/\mathbb Z))
\rightarrow H^3(\Bbb Z/p^s,\mathbb Q/\mathbb Z)$$ is injective, hence the kernel of the map in Theorem \ref{prop:Br} $$u\cup: H^1(\Bbb Z/p^s, H^1(H,\mathbb Q/\mathbb Z))
\rightarrow H^3(\Bbb Z/p^s,\mathbb Q/\mathbb Z)$$ coincides with $H^1(\Bbb Z/p^s, H^1(G_2,\mathbb Q/\mathbb Z))$. The proof of the case (1) then follows from Lemma \ref{lem:map-0} and the case 1) of Lemma~\ref{cor:abel}.

Suppose $p=2$ and $2\nmid[L:L']$. We need to determine the kernel  of the map $$\Res: H^1(\Bbb Z/2^s, H^1(G_2,\mathbb Q/\mathbb Z))\rightarrow H^1(\Bbb Z/2^s,C')$$ which we denote by $\Delta$. We can write $H^1(\Bbb Z/2^s, H^1(G_2,\mathbb Q/\mathbb Z))=B_1\times B_2\times B_3,$ where
\begin{align*}& B_1:=H^1(\Z/2^s, H^1(\Z/2^s,\Q/\Z))^r,\\
	&B_2:=\prod_{i=1}^{r_1} H^1(\Z/2^s,H^1(\Z/2^{e_i},\Q/\Z)),\\
	& B_3:=\prod_{j=1}^{r_2} H^1(\Z/2^s, H^1(\Z/2^{\mu_j},\Q/\Z)).
\end{align*}
We have $B_1\cong (\Z/2^s)^{r},B_2\cong \prod_{i=1}^{r_1}\Z/2^{e_i}$ and $B_3\cong (\Z/2^s)^{r_2}.$

If $\chi \in B_2$, obviously  $2\mid \#\Ker(\chi_h)$ for any $h\in H$ since $e_i<s$, one obtains $B_2\subset \Delta$ by Lemma \ref{lem:map-0}. By a similar argument as in the proof of Proposition \ref{prop:split} for $A_4$ in (\ref{def:A}), we can show $B_3\subset \Delta$; by a similar argument as in the proof of Proposition \ref{prop:split} for $A_1$ in (\ref{def:A}) we can show that $2A_1\subset \Delta$ and $\chi \not \in \Delta$ for any $\chi\in A_1$ has order $2^s$. The proof then follows from the case 1) of Lemma~\ref{cor:abel}.
%The module $C'$ (in Proposition \ref{prop:Br}) is $C_1\times C_2\times H^1(\Bbb Z/2, \Bbb Z/2)^s\otimes \Bbb Z[\Bbb Z/2]$, where $C_1$ is a subgroup of $H^1(\Bbb Z/2, G_1)$ and $$C_2=\{(\chi_1,\chi_2):\chi_1+\chi_2=0,\chi_i\in H^1((\Bbb Z/2)^r, \Bbb Z/2)\text{ for } i=1,2 \}\cong H^1(\Bbb Z/2, \Bbb Z/2)^r.$$ Hence $H^1(\Bbb Z/2,C')=H^1(\Bbb Z/2,C_1)\times H^1(\Bbb Z/2,C_2)=H^1(\Bbb Z/2,C_1)\times H^1(\Bbb Z/2, \Bbb Z/2)^r$.
%Therefore the restriction of $\Res$ to  $H^1(\Bbb Z/2, (\Bbb Z/2)^r)$ is injective and the restriction to  $H^1(\Bbb Z/2, (\Bbb Z/2)^s)$ is a zero map. The proof follows from  Theorem \ref{prop:Br}.
\end{proof}

\begin{rem}
	In fact, if $L'/k$ is  not cyclic, by the case 1) of Lemma~\ref{cor:abel}, we have an exact sequence $$	0\rightarrow H^2(G/H, \Q/\Z) \rightarrow \Br_\mathrm{un}(X)/\Br_0(X) \rightarrow \Lambda \rightarrow 0$$
where $\Lambda$ is a subgroup of $H^1(G/H,H^1(H,\Q/\Z))$ and can be explicitly determined. 	
\end{rem}

\begin{lem}\label{lem:Cor} Let $G=H \rtimes C_2$ is the dihedral group $D_n$, where $H$ is cyclic of order $n$ and $C_2=\Bbb Z/2$. For any character $\chi \in H^1(H,\Bbb Q/\Bbb Z)$, $\Cor_{H/G}(\chi)=0\in H^1(G,\Bbb Q/\Bbb Z)$.
\end{lem}
\begin{proof}	
Note that $G=H C_2$, by \cite[Corollary 1.5.7 and 1.5.8]{NSW},  we have  $$\Res_{G/H}\Cor_{H/G}(\chi)=N_{G/H}(\chi)=\chi+(-\chi)=0, \Res_{G/C_2}\Cor_{H/G}(\chi)=0,$$
which implies $\Cor_{H/G}(\chi)=0$.
\end{proof}

\begin{prop}\label{prop:dehidral}
Let $P(t)$ be an irreducible polynomial. Let $L=k[t]/(P(t))$ and $L'=L\cap K$. Suppose $L'/k$ has order $2$, $K/k$ is Galois with $G=\Gal(K/k)=H \rtimes \Bbb \Gal(L'/k)$ is the dihedral group $D_{\tilde n}$, where $H=\Gal(K/L')$ is cyclic of order $\tilde n$. Let $X$ be the variety (\ref{def:X}) defined by the equation (\ref{eq:norm}). Then,
\begin{enumerate}[(1)]
	\item if $\tilde n$ is odd or $4\mid \tilde n[L:L']$, then $\Br_\mathrm{un}(X)/\Br_0(X)\cong \Bbb Z/\tilde n.$
	
	\item if $\tilde n$ is even and $4\nmid \tilde n[L:L']$, then $\Br_\mathrm{un}(X)/\Br_0(X)\cong \Bbb Z/(\tilde n/2).$
\end{enumerate}
\end{prop}
\begin{proof}Suppose $\tilde n$ is odd. By Lemma \ref{lem:Cor}, the cokernel of the map $H^1(G,\Bbb Q/\Bbb Z)\rightarrow C$ is $\Bbb Z/\tilde n$. Since $G/H$ is cyclic, one obtains $H^2(G/H, \Q/\Z)=0$. Since $H^1(G/H,H^1(H,\Q/\Z))=0$, by (\ref{seq:Br2}) in Theorem \ref{prop:Br}, we have $\Ker[\Sha^2_\omega(\widehat{T})\rightarrow \Sha^2_\omega(\widehat{T'})]=0$. Therefore, $$\Br_\mathrm{un}(X)/\Br_0(X)\cong \Bbb Z/\tilde n.$$
	
In the following, we assume that $\tilde n$ is even. The cokernel of the map $H^1(G,\Bbb Q/\Bbb Z)\rightarrow C$ in (\ref{seq:Br}) is $\Bbb Z/(\tilde n/2)$ by Lemma \ref{lem:Cor}.  Since $G=\Gal(K/k)=H \rtimes \Bbb Z/2$, then the cohomology class $u=0$ which is given by the extension (\ref{extension}), hence the map $u\cup$ in Theorem \ref{prop:Br} is a zero map. 

Suppose $4\nmid \tilde n[L:L']$. Then $v_2(\tilde n)=1$ and $[L:L']$ is odd. Let $\tilde H$ be the unique subgroup of $H$ of order $2$. Then $H^1(\tilde H,\Q/\Z)$ is the $2$-primary part of $H^1(H,\Q/\Z)$.
Let 
\begin{equation}\label{def:c''}
	\tilde C':=\{(\chi_1,\chi_2):\chi_1+\chi_2=0,\chi_i\in  H^1(\tilde H,\Bbb Q/\Bbb Z)\text{ for } i=1,2 \}\cong H^1(\tilde H,\Bbb Q/\Bbb Z).
\end{equation}
The module $\tilde C'$ is the  $2$-primary part of $C'$ in (\ref{def:C'}). Since $G/H$ has order $2$, one obtains $H^1(G/H, C')\cong H^1(G/H, \tilde C')$. But the natural restriction map $H^1(G/H, H^1(H,\Bbb Q/\Bbb Z))\rightarrow H^1(G/H, \tilde C')$  is an isomorphism by (\ref{def:c''}), therefore the map $\Res$ in Theorem \ref{prop:Br} is injective. So  $\Br_\mathrm{un}(X)/\Br_0(X)\cong \Bbb Z/(\tilde n/2).$

Now it remains the case $4\mid \tilde n[L:L']$.  Let $B$ be the unique subgroup of $H^1(H,\Q/\Z)$ of order $2$. In the following we will show $B\otimes \Bbb Z[G/H]\subset C'$. Let $\chi\in B$, if $2\mid [L:L']$, then $[L:L']\chi =0$, hence $B\otimes \Bbb Z[G/H]\subset C'$; if $2\nmid [L:L']$, then $4\mid n$ , it implies $\chi(h)=0$ for any $h\in H$ with $h^2=1$, hence $B\otimes \Bbb Z[G/H]\subset C'$. 

On the other hand, the natural map $H^{-1}(G/H, H^1(H,\Q/\Z)) \rightarrow H^{-1}(G/H, B)$ is an isomorphism. Since $G/H$ is cyclic, so $H^1(G/H, H^1(H,\Q/\Z)) \cong H^1(G/H, B)$,  hence the map $$\Res: H^1(G/H, H^1(H,\Bbb Q/\Bbb Z))\rightarrow H^1(G/H, C')$$ factors through $H^1(G/H, B\otimes \Bbb Z[G/H])=0$.  Therefore $\Ker[\Sha^2_\omega(\widehat{T})\rightarrow \Sha^2_\omega(\widehat{T'})]\cong \Bbb Z/2$ and the cardinality of  $\Br_\mathrm{un}(X)/\Br_0(X)$ is $\tilde n$ by Theorem \ref{prop:Br}. In the following, we will show $\Br_\mathrm{un}(X)/\Br_0(X)\cong \Bbb Z/\tilde n$ by showing that it is a subgroup of a cyclic group.

%
% We can see that $\Bbb Z/2\otimes \Bbb Z[G/H]$ is a subgroup of $C'$. Since $H^1(G/H, H^1(H,\Q/\Z))\cong H^1(G/H,B)= \Z/2$, the morphism $\Res: H^1(G/H, H^1(H,\Q/\Z)) $
%
%
%
%If $2\mid [L:L']$, then $\Res$ is a zero map by Lemma \ref{lem:map-0}.  If $4\mid \tilde n$, then
%the module $C'$ is $\Bbb Z/2\otimes \Bbb Z[G/H]$,  hence $H^1(G/H,C')=0$.
%Therefore $\Res$ is a zero map. Therefore $\Ker[\Sha^2_\omega(\widehat{T})\rightarrow \Sha^2_\omega(\widehat{T'})]\cong \Bbb Z/2$ the cardinality of  $\Br_\mathrm{un}(X)/\Br_0(X)$ is $\tilde n$ by Theorem \ref{prop:Br}. In the following, we will show $\Br_\mathrm{un}(X)/\Br_0(X)\cong \Bbb Z/\tilde n$ by showing it is a subgroup of a cyclic group.

 First, we will show the map $\Br_\mathrm{un}(X)/\Br_0(X) \rightarrow \Br_\mathrm{un}(X_{L'})/\Br_0(X_{L'}) $ is injective, $i.e.$, the map $H^1(k,\Pic(\overline X)) \rightarrow H^1(L',\Pic(\overline X)) $ is injective. It is sufficient to show that $H^1(L'/k,\Pic(\overline X)^{\Gamma_{L'}})=0$.

Let $U'_1$ be the  open affine subvariety of $X$ defined by (\ref{eq:norm}) and $P(t)\neq 0$.
 By the exact sequence  \begin{equation}\label{seq:U'_times} 0\rightarrow \Bbb Z \rightarrow \Bbb Z[L/k]\oplus \Bbb Z[K/k]\rightarrow \bar k[U'_1]^\times/\bar k^\times \rightarrow 0,
 \end{equation}
 it is clear that $ H^1(L',\bar k[U'_1]^\times/\bar k^\times)=0$ since $L\cap K=L'$.

 The exact sequence
 \begin{equation*}
 	0\rightarrow \bar k[U'_1]^\times/\bar k^\times\rightarrow \mathbb D \rightarrow \Pic(\overline X)\rightarrow 0
 \end{equation*}
yields the following exact sequence
 \begin{equation}\label{seq:U'_1}
 	0\rightarrow (\bar k[U'_1]^\times/\bar k^\times)^{\Gamma_{L'}}\rightarrow \mathbb D^{\Gamma_{L'}} \rightarrow \Pic(\overline X)^{\Gamma_{L'}}\rightarrow H^1(L',\bar k[U'_1]^\times/\bar k^\times)= 0,
 \end{equation}
where $\mathbb D$ is defined by (\ref{D}).
Since $\mathbb D^{\Gamma_{L'}}$ is a permutation module, the sequence (\ref{seq:U'_1}) yields the injectivity of the map
$H^1(L'/k,\Pic(\overline X)^{\Gamma_{L'}})\rightarrow H^2(L'/k,(\bar k[U'_1]^\times/\bar k^\times)^{\Gamma_{L'}}).$
By the exact sequence~(\ref{seq:U'_times}), we can show $H^2(L'/k,(\bar k[U'_1]^\times/\bar k^\times)^{\Gamma_{L'}})\cong H^3(L'/k, \Bbb Z)=0$ since $L'/k$ is cyclic, hence $H^1(L'/k,\Pic(\overline X)^{\Gamma_{L'}})=0$.

Now it remains to show that $H^1(L',\Pic(\overline X))$ is cyclic. Let $L\otimes_k L'=L_1\oplus L_2$, by the definition of $\widehat {T'}$, then
$$ H^1(L',\widehat{T'})\subset \oplus_{i=1}^2 H^1(L_i.K/L_i,\mathbb Q/\mathbb Z)\cong \oplus_{i=1}^2 H^1(K/L',\mathbb Q/\mathbb Z).$$
Since $H^1(L',\widehat{T})\cong H^1(K/L',\mathbb Q/\mathbb Z)$, we have
$H^1(L',\widehat{T'})/j^*(H^1(L',\widehat{T}))\subset H^1(K/L',\mathbb Q/\mathbb Z)$ is a cyclic group  by the cyclicity of  $K/L'$. It is clear that $\Sha^2_\omega(\widehat T_{L'})=0$ since $K/L'$ is cyclic,  hence $H^1(L',\Pic(\overline X))$ is cyclic by (\ref{seq:main-1}) (replacing $k$ by $L'$).
\end{proof}
\begin{rem} Let $\Br_\mathrm{vert}(X):= \Br(k(t)) \cap \Br_\mathrm{un}(X)$. In fact, $\Br_\mathrm{vert}(X)/\Br_0(X)$ coincides with the image of the first map in (\ref{seq:Br-Galois}) by \cite[Proposition 2.5-b]{CHS}, then the sequence (\ref{seq:main-1}) may be rewritten as the following exact sequence
	\begin{equation}\label{seq:vert}
		0\rightarrow   \Br_\mathrm{vert}(X)/\Br_0(X) \rightarrow \Br_\mathrm{un}(X)/\Br_0(X)  \rightarrow \Ker[\Sha^2_\omega(\widehat{T})\rightarrow \Sha^2_\omega(\widehat{T'})]\rightarrow 0.
	\end{equation}
If $4\mid \tilde n$,  the case (1) of Proposition \ref{prop:dehidral}  shows that the sequence (\ref{seq:vert}) needs not split.	
\end{rem}

% Let $\left<g\right>$ be a cyclic subgroup of $G'=\Gal(K'/k)$. Let
%$K_{g}$ be the fixed field of $\left<g\right>$ in $K$. We have the
%following morphism
%$$f: H^1(G',\Pic(X_{K'}))\rightarrow \prod\limits_{\left<g\right>}H^1(\left<g\right>,\Pic(X_{K'})).$$
%Let $T''$ be the $k$-torus defined by $N_{L/k}(\Xi_1)\cdot
%N_{K/k}(\Xi_2)=1$ and $\mathbb T$ the $k$-torus defined by $N_{L'/k}(\Xi)=1$. Since
%$$\Ker(f)\cong \Sha^2_\omega (\widehat {T''} )\cong\Sha^2_\omega (K/k,\widehat {\mathbb T} )=0$$
%by \cite[Proposition 2.3]{Wei14} and \cite[Theorem 6]{DW}. If $ $

%\begin{lem}\label{lem:3.2} Suppose $P(t)$ is a separable polynomial and $K/k$ is Galois. we have $$\Ker[\Sha^2_\omega(\widehat{T})\rightarrow \Sha^2_\omega(\widehat{T'})]=\Ker[\Sha^2_\omega(\widehat{T})\rightarrow \Sha^2_\omega(\widehat{\mathbf T''})].$$
%\end{lem}
%\begin{proof} By Lemma \ref{lem:3.1}, we only need to show the map $\Sha^2_\omega(\widehat{\mathbf T'})\rightarrow \Sha^2_\omega(\widehat{\mathbf T''})$ is injective. We have the exact sequence $$0\rightarrow M \rightarrow \widehat{\mathbf T'} \rightarrow \widehat{\mathbf T''}\rightarrow 0,$$
%where $M$ is a permutation $G$-module. For any $g\in G$, the above exact sequence give the exact sequence.
%$$H^1(<g>,\widehat{\mathbf T'} )\rightarrow H^1(<g>,\widehat{\mathbf T''} )\rightarrow H^2(<g>,M )\rightarrow H^2(<g>,\widehat{\mathbf T'} )\rightarrow H^2(<g>,\widehat{\mathbf T''} )$$
%
%\end{proof}

\bibliographystyle{alpha}
%\bibliography{mybib1}

\begin{thebibliography}{99}

\bibitem{BHB11} T. D. Browning and D. R. Heath-Brown, \emph{Quadratic polynomials represented
by norm forms}, Geom. Funct. Anal. 22 (2012), 1124--1190.

\bibitem{BM} T. D. Browning and L. Matthiesen, \emph{Norm forms for arbitrary number fields as products of linear polynomials}, Ann. Sci. \'Ecole Norm. Sup. (4) 50 (2017), 1383--1446.

\bibitem{BMS14} T. D. Browning, L. Matthiesen and A. N. Skorobogatov, \emph{Rational points on pencils of conics and quadrics with many degenerate fibers}, Ann. of Math. (2) 180 (2014), 381--402.

\bibitem{CWX18} Y. Cao, D. Wei and F. Xu, \emph{Strong approximation for a family of norm varieties}, arXiv:1803.11003v5.

\bibitem{CT03} J.-L. Colliot-Th{\'e}l{\`e}ne, \emph{Points rationnels sur les fibrations}, in Higher dimensional varieties and rational points (Budapest, 2001), volume 12 of Bolyai Soc. Math. Stud., pages 171--221. Springer, Berlin, 2003.

%\bibitem{CH16} J.-L. Colliot-Th{\'e}l{\`e}ne and D. Harari,\emph{Approximation forte en famille}, J. Reine Angew. Math. 710 (2016), 173--198

\bibitem{CHS} J.-L. Colliot-Th\'el\`ene, D. Harari and A. N. Skorobogatov, \emph{Valeurs d'un polyn\^{o}me \`a une variable repr\'esent\'ees par une norme}, in  "Number Theory and Algebraic Geometry", ed. Miles Reid and Alexei Skorobogatov,  London Math. Soc. Lecture Note Ser., 303, 2003, 69--89.

\bibitem{CHS1} J.-L. Colliot-Th\'el\`ene, D. Harari and A. N. Skorobogatov, \emph{Compactification \'equivariante d'un tore (d'apr\`es Brylinski et K\"unnemann)},   Expo. Math.  23 (2005), 161--170.

\bibitem{CTS89} J.-L. Colliot-Th{\'e}l{\`e}ne and P. Salberger, \emph{Arithmetic on some singular cubic hypersurfaces}, Proc. London Math. Soc. (3) 58 (1989), 519--549.

\bibitem{CTS87} J.-L. Colliot-Th{\'e}l{\`e}ne and J.-J. Sansuc, \emph{La descente sur les vari\'et\'es rationnelles} II, Duke Math. J. 54 (1987), 375--492.

\bibitem{CTSSD87a} J.-L. Colliot-Th{\'e}l{\`e}ne, J.-J. Sansuc and P. Swinnerton-Dyer, \emph{Intersections of two quadrics and Ch\^atelet surfaces} I, J. Reine Angew. Math. 373 (1987), 37--107.
	
\bibitem{CTSSD87b} J.-L. Colliot-Colliot-Th{\'e}l{\`e}ne, J.-J. Sansuc and P. Swinnerton-Dyer, \emph{Intersections of two quadrics and Ch\^atelet surfaces} II, J. Reine Angew. Math. 374 (1987), 72--168.

\bibitem{CTSSD98} J.-L. Colliot-Th{\'e}l{\`e}ne, A. N. Skorobogatov and P. Swinnerton-Dyer,  \emph{Rational points and zero-cycles on fibred varieties: Schinzel's hypothesis and Salberger's device}, J. Reine Angew. Math. 495 (1998), 1--28.

\bibitem{CTSD94}J.-L. Colliot-Th{\'e}l{\`e}ne and P. Swinnerton-Dyer, \emph{Hasse principle and weak approximation for
pencils of Severi-Brauer and similar varieties}, J. Reine Angew. Math. 453 (1994), 49--112.

\bibitem{DSW14}  U. Derenthal, A. Smeets and D. Wei, \emph{Universal torsors and values of quadratic polynomials represented by norms}, Math. Ann. 361 (2014), 1021--1042.

\bibitem{DW17} U.~Derenthal and D.~Wei, \emph{Strong approximation and descent},
J. Reine Angew. Math. 731 (2017), 235--258.
			
%\bibitem{Has30} H. Hasse. \emph{Die Normenresttheorie relativ-Abelscher Zahlk$\ddot{o}$rper als
%			Klassenk$\ddot{o}$rpertheorie im Kleinen}, J. f. M. 162 (1930), 145--154.


\bibitem{HSW} Y. Harpaz, A. N. Skorobogatov and O. Wittenberg, \emph{The Hardy-Littlewood conjecture and rational points}, Compos. Math. 150 (2014), 2095--2111.

\bibitem{Ha} R. Hartshorne, \emph{Algebraic geometry}, GTM, vol. 52, Springer-Verlag, New York, 1977.
				
\bibitem{HBS02} D. R. Heath-Brown and A. Skorobogatov, \emph{Rational solutions of certain equations involving norms}, Acta Math. 189 (2002), 161--177.

\bibitem{Jan90} U. Jannsen. \emph{The splitting of the Hochschild-Serre spectral sequence for a product of groups},  Canad. Math. Bull. 33 (1990), 181--183.

\bibitem{Li69} S. Lichtenbaum, \emph{Duality theorems for curves over p-adic fields}, Invent. Math. 7 (1969), 120--136.

\bibitem{NSW}  J. Neukirch, A. Schmidt and K. Wingberg,  \emph{Cohomology of Number Fields},  Grundlehren der Math. 323, Springer,  2000.

\bibitem{San81} J.-J. Sansuc, \emph{Groupe de Brauer et arithm\'etique des groupes alg\'ebriques lin\'eaires sur un corps de nombres}, J. Reine Angew. Math. 327 (1981), 12--80.
						
\bibitem{SJ11} M. Swarbrick Jones, \emph{A note on a theorem of Heath-Brown and Skorobogatov}, Q. J. Math. 64 (2013), 1239--1251.


\bibitem{VV12} A. V\'arilly-Alvarado and B. Viray, \emph{Higher-dimensional analogs of Chatelet surfaces}, Bull. Lond. Math. Soc. 44 (2012), 125--135. 

\bibitem{VV15} A. V\'arilly-Alvarado and B. Viray, \emph{Smooth compactifications of certain normic bundles}, Eur. J. Math. 1 (2015), 250--259. 

\bibitem{Wei14} D. Wei, \emph{On the equation $N_{K/k}(\Xi) = P(t)$}, Proc. Lond. Math. Soc. (3) 109 (2014), 1402--1434.

\bibitem{Wit18} O. Wittenberg, \emph{Rational points and zero-cycles on rationally connected varieties over number fields}, Algebraic geometry: Salt Lake City 2015, 597--635, Proc. Sympos. Pure Math., 97, Amer. Math. Soc., Providence, RI, 2018.






\end{thebibliography}
\end{document}